\documentclass[3p]{elsarticle}

\usepackage{amssymb}
\usepackage{amsfonts}
\usepackage{amsmath}
\usepackage{amsthm}
\usepackage{graphicx}
\usepackage{mathrsfs}
\usepackage{dsfont}
\usepackage{amscd}
\usepackage{multirow}
\usepackage[all]{xy}
\normalfont
\usepackage[T1]{fontenc}
\usepackage{calligra}
\usepackage{verbatim}
\usepackage[usenames,dvipsnames]{color}
\usepackage[colorlinks=true,linkcolor=Blue,citecolor=Violet]{hyperref}
\usepackage{enumerate}

\newcommand{\N}{{\mathds{N}}}

\newcommand{\R}{{\mathds{R}}}

\newcommand{\D}{{\mathfrak{D}}}
\newcommand{\A}{{\mathfrak{A}}}
\newcommand{\B}{{\mathfrak{B}}}

\newcommand{\bigslant}[2]{{\raisebox{.2em}{$#1$}\left/\raisebox{-.2em}{$#2$}\right.}}

\newcommand{\Lip}{{\mathsf{L}}}

\newcommand{\dpropinquity}[1]{{\mathsf{\Lambda}^\ast_{#1}}}
\newcommand{\covpropinquity}[1]{{\mathsf{\Lambda}^{\mathrm{cov}}_{#1}}}

\newcommand{\Kantorovich}[1]{{\mathsf{mk}_{#1}}}

\newcommand{\KantorovichDist}[3]{{\mathrm{mkD}_{#1}\left({#2},{#3}\right)}}

\newcommand{\Haus}[1]{{\mathsf{Haus}_{#1}}}

\newcommand{\StateSpace}{{\mathscr{S}}}

\newcommand{\MongeKant}{{Mon\-ge-Kan\-to\-ro\-vich metric}}

\newcommand{\qcms}{quantum compact metric space}

\newcommand{\unit}{1}

\newcommand{\sa}[1]{{\mathfrak{sa}\left({#1}\right)}}

\newcommand{\UIso}[4]{{\mathsf{UIso}_{#1}\left({#2}\rightarrow{#3}\middle\vert{#4}\right)}}

\newcommand{\dom}[1]{{\operatorname*{dom}\left({#1}\right)}}

\newcommand{\norm}[2]{{\left\|{#1}\right\|_{#2}}}
\newcommand{\tunnelset}[4]{{\text{\calligra Tunnels}\,\left[{#1}\stackrel{#3}{\longrightarrow}{#2}\middle\vert {#4} \right]}}

\newcommand{\targetsettunnel}[3]{{\mathfrak{t}_{#1}\left({#2}\middle\vert{#3}\right)}}

\newcommand{\targetsetimage}[3]{{\mathfrak{i}_{#1}\left({#2}\middle\vert{#3}\right)}}
\newcommand{\targetsetforward}[3]{{\mathfrak{f}_{#1}\left({#2}\middle\vert{#3}\right)}}

\newcommand{\worknote}[1]{} 
\newcommand{\opnorm}[3]{{\left|\mkern-1.5mu\left|\mkern-1.5mu\left| {#1} \right|\mkern-1.5mu\right|\mkern-1.5mu\right|_{#3}^{#2}}}

\newcommand{\tunnelreach}[2]{{\rho\left({#1}\middle\vert{#2}\right)}}

\newcommand{\tunnelmagnitude}[2]{{\mu\left({#1}\middle\vert{#2}\right)}}

\newcommand{\tunnelextent}[1]{{\chi\left({#1}\right)}}

\newcommand{\dil}[1]{{\mathrm{dil}\left({#1}\right)}}

\theoremstyle{plain}
\newtheorem{theorem}{Theorem}[section]

\newtheorem{corollary}[theorem]{Corollary}

\newtheorem{lemma}[theorem]{Lemma}

\newtheorem{theorem-definition}[theorem]{Theorem-Definition}

\theoremstyle{definition}
\newtheorem{definition}[theorem]{Definition}

\newtheorem{convention}[theorem]{Convention}

\theoremstyle{remark}

\newtheorem{remark}[theorem]{Remark}

\newtheorem{notation}[theorem]{Notation}

\renewcommand{\geq}{\geqslant}
\renewcommand{\leq}{\leqslant}

\numberwithin{equation}{section}

\allowdisplaybreaks[4]

\title{Convergence of Cauchy Sequences for the Covariant Gromov-Hausdorff Propinquity}
\author{Fr\'{e}d\'{e}ric Latr\'{e}moli\`{e}re}
\ead{frederic@math.du.edu}
\ead[url]{http://www.math.du.edu/\symbol{126}frederic}
\address{Department of Mathematics \\ University of Denver \\ Denver CO 80208}

\begin{document}

\begin{abstract}
  The covariant Gromov-Hausdorff propinquity is a distance on Lipschitz dynamical systems over quantum compact metric spaces, up to equivariant full quantum isometry. It is built from the dual Gromov-Hausdorff propinquity which, as its classical counterpart, is complete. We prove in this paper several sufficient conditions for convergence of Cauchy sequences for the covariant propinquity and apply it to show that many natural classes of dynamical systems are complete for this metric.
\end{abstract}

\maketitle


\section{Introduction}

The covariant Gromov-Hausdorff propinquity is a distance, up to equivariant full quantum isometry, on the class of Lipschitz dynamical systems, defined as the class of {\qcms} endowed with a strongly continuous action of a proper monoid by Lipschitz morphisms or even Lipschitz linear maps. The class of Lipschitz dynamical systems include C*-dynamical systems by Lipschitz automorphisms as well as actions by monoids of completely positive maps which map the domain of the L-seminorm of a {\qcms} to itself. We proved in \cite{Latremoliere18b} that the covariant propinquity is a metric up to equivariant full quantum isometry --- namely, distance zero implies the existence of a full quantum isometry between the {\qcms s} as well as an isometric isomorphism between the acting monoids, which intertwine the actions in a natural fashion. We illustrate in \cite{Latremoliere18b} our metric by showing that fuzzy tori with their dual actions converge to quantum tori with their own dual actions for the covariant propinquity. 

The covariant propinquity is built from the dual Gromov-Hausdorff propinquity \cite{Latremoliere13,Latremoliere13b,Latremoliere13c,Latremoliere14,Latremoliere15,Latremoliere15b}, which actually enjoys some natural covariance properties \cite{Latremoliere17c}, though it is only defined on {\qcms s} and thus does not fully capture the structure of a Lipschitz dynamical system. Nonetheless, our work in \cite{Latremoliere17c} suggests that a covariant propinquity, as introduced in \cite{Latremoliere18b}, is a natural object to construct. The covariant propinquity between two Lipschitz dynamical systems dominate the propinquity between the underlying {\qcms s} and the pointed Gromov-Hausdorff distance \cite{Gromov81} between the underlying proper monoids, and both these last two distances are in particular complete. We are thus left with a very natural question: what classes of Lipschitz dynamical systems are complete when endowed with the covariant propinquity?

This question is the subject of the present paper. As the covariant propinquity is built using a covariant version of the pointed Gromov-Hausdorff distance between proper monoids, we begin with finding natural classes of proper monoids complete for the monoid-adapted Gromov-Hausdorff distance. We then discover that completeness is not a trivial matter, and in fact, we require a form of equicontinuity of the right translations of our monoids to provide a sufficient condition on Cauchy sequences to converge. We also see that additional complications arise when working with proper groups. We are however able to establish a generous sufficient condition which applies to a large class of natural examples, especially arising from actions on {\qcms s}.

We then turn to the matter of convergence for Cauchy sequences for the covariant propinquity. We provide a sufficient condition which includes the condition exhibited for the monoid-Gromov-Hausdorff distance, and a similar condition on the actions themselves. The reason for this double condition is simply that we actually allow for a metric on proper monoids of Lipschitz dynamical systems which may not be the same as the natural pseudo-metric induced by the quantum compact metric space structure of the space on which the monoid acts. We explain this matter toward the end of this paper. We are in fact able to prove a stronger result than completeness for some classes of Lipschitz dynamical systems: we exhibit a sufficient condition for sequential compactness of certain classes of Lipschitz dynamical systems based upon the inherent covariant properties of the dual propinquity. From this result, our completeness result derives. We do know of a direct proof of our completeness result which does not involve the compactness properties we prove in this paper, but our proof of completeness does not weaken the assumptions we make here, and thus this approach is the most potent we know at this moment.

We begin our paper with a background section on noncommutative metric geometry and the covariant propinquity to set up the framework of this paper.

\emph{Acknowledgment: }This work is part of the project supported by the grant H2020-MSCA-RISE-2015-691246-QUANTUM DYNAMICS and grant \#3542/H2020/2016/2 of the Polish Ministry of Science and Higher Education.

\section{The covariant Gromov-Hausdorff Propinquity}

The covariant propinquity is defined on Lipschitz dynamical systems, which are proper monoid actions on {\qcms s} by Lipschitz maps. A {\qcms} is a noncommutative analogue of the algebra of Lipschitz functions over a compact metric space \cite{Connes89,Rieffel98a,Rieffel99,Latremoliere05b,Latremoliere12b,Latremoliere13}.

\begin{notation}
  Throughout this paper, for any unital C*-algebra $\A$, the norm of $\A$ is denoted by $\norm{\cdot}{\A}$, the space of self-adjoint elements in $\A$ is denoted by $\sa{\A}$, the unit of $\A$ is denoted by $\unit_\A$ and the state space of $\A$ is denoted by $\StateSpace(\A)$. We also adopt the convention that if a seminorm $\Lip$ is defined on some dense subspace of $\sa{\A}$ and $a\in\sa{\A}$ is not in the domain of $\Lip$, then $\Lip(a) = \infty$.
\end{notation}

\begin{definition}\label{qcms-def}
  A \emph{\qcms} $(\A,\Lip)$ is an ordered pair of a unital C*-algebra $\A$ and a seminorm $\Lip$, called an \emph{L-seminorm}, defined on a dense Jordan-Lie subalgebra $\dom{\Lip}$ of $\sa{\A}$, such that:
  \begin{enumerate}
    \item $\{ a \in \sa{\A} : \Lip(a) = 0 \} = \R\unit_\A$,
    \item the \emph{\MongeKant} $\Kantorovich{\Lip}$ defined for any two states $\varphi, \psi \in \StateSpace(\A)$ by:
      \begin{equation*}
        \Kantorovich{\Lip}(\varphi, \psi) = \sup\left\{ |\varphi(a) - \psi(a)| : a\in \dom{\Lip}, \Lip(a) \leq 1 \right\}
      \end{equation*}
      metrizes the weak* topology restricted to $\StateSpace(\A)$,
    \item $\Lip$ satisfies the $F$-quasi-Leibniz inequality, i.e. for all $a,b \in \dom{\Lip}$:
      \begin{equation*}
        \max\left\{ \Lip\left(\frac{a b + b a}{2}\right), \Lip\left(\frac{a b - b a}{2i}\right) \right\} \leq F(\norm{a}{\A},\norm{b}{\A},\Lip(a),\Lip(b))\text{,}
      \end{equation*}
      for some \emph{permissible} function $F$, i.e. a function $F : [0,\infty)^4\rightarrow [0,\infty)$, increasing when $[0,\infty)^4$ is endowed with the product order, and such that for all $x,y,l_x,l_y \geq 0$ we have $F(x,y,l_x,l_y) \geq x l_y + y l_x$,
    \item $\Lip$ is lower semi-continuous with respect to $\norm{\cdot}{\A}$.
  \end{enumerate}

We say that $(\A,\Lip)$ is \emph{Leibniz} when $F$ can be chosen to be $F:x,y,l_x,l_y \mapsto x l_y + y l_x$. More generally, if $\Lip$ satisfies the $F$-quasi-Leibniz inequality for some $F$ then $(\A,\Lip)$ is called a {$F$-\qcms}.
\end{definition}

We refer to \cite{Rieffel98a, Rieffel01, Rieffel02, li03, Latremoliere15c, Latremoliere16, Latremoliere15b,Latremoliere15d,Rieffel15b,Aguilar18} for various examples of {\qcms s}, including quantum tori, certain group C*-algebras, AF algebras, noncommutative solenoids, Podles spheres, and more. 
 
Quantum compact metric spaces form a category for the appropriate choices of morphisms. We refer to \cite{Latremoliere16b} for some observations on the definition of Lipschitz morphisms and some of their applications. The definition of quantum isometry relies on a key observation of Rieffel in \cite{Rieffel00}.

\begin{definition}
  Let $(\A,\Lip_\A)$ and  $(\B,\Lip_\B)$ be {\qcms s}.
  \begin{itemize}
  \item A positive unital linear map $\pi : \A \rightarrow \B$ is \emph{Lipschitz} when there exists $k\geq 0$ such that $\Lip_\B\circ\pi\leq k \Lip_\A$.
  \item A \emph{Lipschitz morphism} $\pi : \A \rightarrow \B$ is a unital *-morphism from $\A$ to $\B$ when there exists $k\geq 0$ such that $\Lip_\B\circ\pi\leq k \Lip_\A$.
  \item A \emph{quantum isometry} $\pi : (\A,\Lip_\A) \rightarrow (\B,\Lip_\B)$ is a *-epimorphism from $\A$ onto $\B$ such that for all $b \in \dom{\Lip_\A}$:
    \begin{equation*}
      \Lip_\B(b) = \inf\left\{ \Lip_\A(a) : \pi(a) = b \right\} \text{.}
    \end{equation*}
  \item A \emph{full quantum isometry} $\pi : (\A,\Lip_\A) \rightarrow (\B,\Lip_\B)$ is a *-isomorphism from $\A$ onto $\B$ such that $\Lip_\B\circ\pi = \Lip_\A$.
  \end{itemize}
\end{definition}

We now equip {\qcms s} with actions of proper monoids. As a matter of definition, we recall:
\begin{definition}
  A \emph{metric monoid} $(G,\delta)$ (resp. group) is a monoid (resp. a group) $G$ and a left invariant metric $\delta$ on $G$ for which the multiplication is continuous (resp. the multiplication and the inverse function are continuous).

The metric monoid (resp. group) is \emph{proper} when all its closed balls are compact.
\end{definition}

\begin{definition}
  A \emph{(metric monoid) morphism} $\pi : G \rightarrow H$ is a map such that:
  \begin{itemize}
    \item $\pi$ maps the identity element of $G$ to the identity element of $H$,
    \item $\forall g,h \in G \quad \pi(g h) = \pi(g) \pi(h)$,
    \item $\pi$ is continuous.
  \end{itemize}
\end{definition}

\begin{remark}
  A proper metric space is always complete and separable.
\end{remark}

We now formally define the objects of the space under consideration in this paper: Lipschitz dynamical systems.

\begin{notation}
  Let $(\A,\Lip_\A)$ and $(\B,\Lip_\B)$ be two {\qcms s}. If $\pi : \A \rightarrow \B$ is a unital positive linear map, then:
  \begin{equation*}
    \dil{\pi} = \inf\left\{ k > 0 : \forall a \in \sa{\A} \quad \Lip_\B\circ\pi(a) \leq k \Lip_\A(a) \right\} \text{.}
  \end{equation*}
  By definition, $\dil{\pi} < \infty$ if and only if $\pi$ is a Lipschitz linear map.
\end{notation}

\begin{definition}[{\cite{Latremoliere18b}}]
  Let $F$ be a permissible function. A \emph{Lipschitz dynamical $F$-system} $(\A,\Lip,G,\delta,\alpha)$ is a {$F$-\qcms} $(\A,\Lip)$ and a proper monoid $(G,\delta)$, together with an action $\alpha$ by positive unital maps (i.e. a morphism from $G$ to the monoid of positive linear maps) such that:
\begin{enumerate}
  \item $\alpha$ is strongly continuous: for all $a\in\A$ and $g \in G$, we have:
    \begin{equation*}
      \lim_{h\rightarrow g} \norm{\alpha^h(a) - \alpha^g(a)}{\A} = 0\text{,}
    \end{equation*}
  \item $g\in G\mapsto \dil{\alpha^g}$ is locally bounded: for all $g\in G$ there exist $D > 0$ and a neighborhood $U$ of $g$ in $G$ such that if $h\in U$ then $\dil{\alpha^h} \leq D$.
\end{enumerate}

A \emph{Lipschitz $C^\ast$-dynamical $F$-system} $(\A,\Lip,G,\delta,\alpha)$ is a Lipschitz dynamical system where $G$ is a proper group and $\alpha^g$ is a Lipschitz unital *-automorphism for all $g \in G$.
\end{definition}

The class of Lipschitz dynamical systems include various sub-classes of interest, from group actions by full quantum isometries, to actions by completely positive maps, to actions by Lipschitz automorphisms or even unital endomorphisms.

There is a natural manner to combine the notions of Lipschitz morphisms and Lipschitz morphism of proper monoids into a notion of morphism for Lipschitz dynamical systems. For our purpose, we will focus on what it means for two such systems to be considered the same system, i.e. our notion of isomorphism.

\begin{definition}[{\cite{Latremoliere18b}}]
  Let $\mathds{A} = (\A,\Lip_\A,G,\delta_G,\alpha)$ and $\mathds{B} = (\B,\Lip_\B,H,\delta_H,\beta)$ be two Lipschitz dynamical systems. An \emph{equivariant quantum full isometry} $(\pi,\varsigma) : \mathds{A} \rightarrow \mathds{B}$ is given by a full quantum isometry $\pi : (\A,\Lip_\A) \rightarrow (\B,\Lip_\B)$ and a monoid isometric isomorphism $\varsigma : G\rightarrow H$ such that for all $g \in G$:
\begin{equation*}
  \pi\circ\alpha^g = \beta^{\varsigma(g)}\circ\pi \text{.}
\end{equation*}
\end{definition}

The construction of the covariant propinquity begins with the definition of a monoid-adapted Gromov-Hausdorff distance. We define our distance between two proper metric monoids $(G_1,\delta_1)$ and $(G_2,\delta_2)$ by measuring how far a given pair of maps $\varsigma_1:G_1\rightarrow G_2$ and $\varsigma_2 : G_2\rightarrow G_1$ is from being an isometric isomorphism and its inverse. There are at least two ways to do so. The following definition will serve this purpose well for us.

\begin{notation}
  For a metric space $(X,\delta)$, $x\in X$ and $r\geq 0$, the closed ball in $(X,\delta)$ centered at $x$, of radius $r$, is denoted as $X_\delta[x,r]$, or simply $X[x,r]$.
  If $(G,\delta)$ is a metric monoid with identity element $e \in G$, and if $r \geq 0$, then $G[e,r]$ is denoted as $G[r]$. 
\end{notation}

\begin{definition}[{\cite{Latremoliere18b}}]\label{almost-iso-def}
Let $(G_1,\delta_1)$ and $(G_2,\delta_2)$ be two metric monoids with respective identity elements $e_1$ and $e_2$. An \emph{$r$-local $\varepsilon$-almost isometric isomorphism} $(\varsigma_1,\varsigma_2)$, for $\varepsilon \geq 0$ and $r \geq 0$,  is an ordered pair of maps $\varsigma_1 : G_1 \rightarrow G_2$ and $\varsigma_2 : G_2 \rightarrow G_1$ such that for all $\{j,k\} = \{1,2\}$:
\begin{equation*}
\forall g,g' \in G_j[r] \quad \forall h \in G_k[r] \quad \left| \delta_k(\varsigma_j(g)\varsigma_j(g'),h) - \delta_j(gg',\varsigma_k(h))\right| \leq \varepsilon\text{,}
\end{equation*}
and
\begin{equation*}
\varsigma_j(e_j) = e_k \text{.}
\end{equation*}
 The set of all $r$-local $\varepsilon$-almost isometric isomorphism is denoted by:
\begin{equation*}
  \UIso{\varepsilon}{(G_1,\delta_1)}{(G_2,\delta_2)}{r} \text{.}
\end{equation*}
\end{definition}

\begin{convention}
  Write $f_{\big|D}$ for the restriction of a function $f$ to some subset $D$ of its domain. Let $\varsigma : D_1 \subseteq G_1 \rightarrow G_2$ and $\varkappa : D_2 \subseteq G_2 \rightarrow G_1$ with $G_j[r] \subseteq D_j$ for some $r \geq 0$ and $j \in \{1,2\}$. For any $\varepsilon \geq 0$, we will simply write $(\varsigma,\varkappa) \in \UIso{\varepsilon}{(G_1,\delta_1)}{(G_2,\delta_2)}{r}$ to mean:
  \begin{equation*}
    \left(\varsigma_{\big|G_1[r]},\varkappa_{\big| G_2[r]}\right) \in \UIso{\varepsilon}{(G_1,\delta_1)}{(G_2,\delta_2)}{r} \text{.}
  \end{equation*}

Moreover, if $(\varsigma,\varkappa) \in \UIso{\varepsilon}{(G_1,\delta_1)}{(G_2,\delta_2)}{r}$ then we may as well assume that $\varsigma$ and $\varkappa$ are defined on $G_1$ and $G_2$, respectively, by choosing any extension of $\varsigma$ and $\varkappa$, since it does not affect the local almost isometry property.
\end{convention}
We record that almost isometries behave well under composition.

\begin{lemma}[{\cite{Latremoliere18b}}]\label{uiso-composition-lemma}
Let $(G_1,\delta_1)$, $(G_2,\delta_2)$ and $(G_3,\delta_3)$ be three metric monoids with respective identity elements $e_1$, $e_2$ and $e_3$. 

If:
\begin{align*}
  (\varsigma_1,\varkappa_1)\in\UIso{\varepsilon_1}{(G_1,\delta_1)}{(G_2,\delta_2)}{\frac{1}{\varepsilon_1}}
\intertext{ and } 
(\varsigma_2,\varkappa_2)\in\UIso{\varepsilon_2}{(G_2,\delta_2)}{(G_3,\delta_3)}{\frac{1}{\varepsilon_2}}
\end{align*}
for some $\varepsilon_1, \varepsilon_2 \in \left( 0, \frac{\sqrt{2}}{2} \right]$, then:
\begin{equation*}
\left(\varsigma_2\circ\varsigma_1,\varkappa_1\circ\varkappa_2\right) \in \UIso{\varepsilon_1 + \varepsilon_2}{(G_1,\delta_1)}{(G_3,\delta_3)}{\frac{1}{\varepsilon_1 + \varepsilon_2}} \text{.}
\end{equation*}
\end{lemma}

Our covariant Gromov-Hausdorff distance over the class of proper metric monoids is then defined along the lines Gromov's distance.

\begin{definition}[{\cite{Latremoliere18b}}]\label{group-GH-def}
The \emph{Gromov-Hausdorff monoid distance} $\Upsilon((G_1,\delta_1),(G_2,\delta_2))$ between two proper metric monoids $(G_1,\delta_1)$ and $(G_2,\delta_2)$ is given by:
  \begin{equation*}
    \Upsilon((G_1,\delta_1),(G_2,\delta_2)) = \min\left\{ \frac{\sqrt{2}}{2}, \inf\left\{ \varepsilon > 0 \middle\vert \UIso{\varepsilon}{(G_1,\delta_1)}{(G_2,\delta_2)}{\frac{1}{\varepsilon}} \not= \emptyset \right\}\right\} \text{.}
  \end{equation*}
\end{definition}

We indeed prove in \cite{Latremoliere18b}.

\begin{theorem}[{\cite{Latremoliere18b}}]\label{upsilon-metric-thm}
For any proper metric monoids $(G_1,\delta_1)$, $(G_2,\delta_2)$ and $(G_3,\delta_3)$:
\begin{enumerate}
\item $\Upsilon((G_1,\delta_1),(G_2,\delta_2)) \leq \frac{\sqrt{2}}{2}$,
\item $\Upsilon((G_1,\delta_1),(G_2,\delta_2)) = \Upsilon((G_2,\delta_2),(G_1,\delta_1))$,
\item $\Upsilon((G_1,\delta_1),(G_3,\delta_3)) \leq \Upsilon((G_1,\delta_1),(G_2,\delta_2)) + \Upsilon((G_2,\delta_2),(G_3,\delta_3))$,
\item If $\Upsilon((G_1,\delta_1),(G_2,\delta_2)) = 0$ if and only if there exists a monoid isometric isomorphism from $(G_1,\delta_1)$ to $(G_2,\delta_2)$.
\end{enumerate}
In particular, $\Upsilon$ is a metric up to metric group isometric isomorphism on the class of proper metric groups.

Moreover, if $\mathrm{GH}$ is the pointed Gromov-Hausdorff distance on proper metric spaces, and if $e_1$ and $e_2$ are the respective identity elements of $G_1$ and $G_2$, then:
\begin{equation*}
  \mathrm{GH}((G_1,\delta_1,e_1),(G_2,\delta_2,e_2)) \leq \Upsilon((G_1,\delta_1),(G_2,\delta_2)) \text{.}
\end{equation*}
\end{theorem}

It will also be helpful to recall from \cite{Latremoliere18b} the following properties of almost isometries. 

\begin{lemma}[{\cite{Latremoliere18b}}]\label{almost-isoiso-lemma}
Let $(G_1,\delta_1)$, $(G_2,\delta_2)$ be two metric monoids and $\varepsilon \geq 0$, $r > 0$. If $(\varsigma_1,\varsigma_2) \in \UIso{\varepsilon}{(G_1,\delta_1)}{(G_2,\delta_2)}{r}$ then for all $\{j,k\} = \{1,2\}$, if $r'=\max\{0,r-\varepsilon\}$ then:
\begin{enumerate}
\item $\forall g \in G_j[r] \quad \forall h \in G_k[r] \quad \left| \delta_k(\varsigma_j(g),h) - \delta_j(g,\varsigma_k(h)) \right| \leq \varepsilon$,
\item $\forall t \in [0,r] \quad \forall g \in G_j[t] \quad \varsigma_j(g) \in G_k[t + \varepsilon]$,
\item $\forall g \in G_j[r'] \quad \delta_j(\varsigma_k\circ\varsigma_j(g),g) \leq \varepsilon$,
\item $\forall g,g' \in G_j\left[\frac{r'}{2}\right] \quad \delta_k(\varsigma_j(g)\varsigma_j(g'),\varsigma_j(gg')) \leq 2\varepsilon$,
\item $\forall g,g' \in G_j[r'] \quad \left|\delta_k(\varsigma_j(g),\varsigma_j(g')) - \delta_j(g,g')\right| \leq 2\varepsilon$;
\end{enumerate}
\end{lemma}

The construction of the covariant propinquity begins with generalizing the notion of a tunnel between {\qcms s}, as defined in \cite{Latremoliere13b,Latremoliere14} for our construction of the Gromov-Hausdorff propinquity, to our class of Lipschitz dynamical systems. Notably, the needed changes are minimal.

\begin{definition}[{\cite{Latremoliere18b}}]\label{equi-tunnel-def}
Let $\varepsilon > 0$ and $F$ be a permissible function. Let $(\A_1,\Lip_1,G_1,\delta_1,\alpha_1)$ and $(\A_2,\Lip_2,G_2,\delta_2,\alpha_2)$ be two Lipschitz dynamical $F$-systems. Let $e_1$ and $e_2$ be the identity elements of $G_1$ and $G_2$ respectively. A \emph{$\varepsilon$-covariant $F$-tunnel}:
\begin{equation*}
\tau = (\D,\Lip_\D,\pi_1,\pi_2,\varsigma_1,\varsigma_2)
\end{equation*}
from $(\A_1,\Lip_1,G_1,\delta_1,\alpha_1)$ to $(\A_2,\Lip_2,G_2,\delta_2,\alpha_2)$ is given by 
\begin{equation*}
  (\varsigma_1,\varsigma_2) \in \UIso{\varepsilon}{(G_1,\delta_1)}{(G_2,\delta_2)}{\frac{1}{\varepsilon}} \text{,}
\end{equation*}
an $F$-{\qcms} $(\D,\Lip_\D)$, and two quantum isometries $\pi_1 : (\D,\Lip_\D) \twoheadrightarrow (\A_1,\Lip_1)$ and $\pi_2 : (\D,\Lip_\D) \twoheadrightarrow (\A_2,\Lip_2)$.
\end{definition}

\begin{remark}
  If $\tau$ is an $\varepsilon$-covariant tunnel then it is also an $\eta$-covariant tunnel for any $\eta \geq \varepsilon$.
\end{remark}

\begin{remark}
  If $(\D,\Lip,\pi,\rho,\varsigma,\varkappa)$ is a covariant tunnel from $(\A,\Lip_\A,G,\delta_G,\alpha)$ to $(\B,\allowbreak \Lip_\B,H,\delta_H,\beta)$, then $(\D,\Lip,\pi,\rho)$ is a tunnel from $(\A,\Lip_\A)$ to $(\B,\Lip_\B)$ in the sense of \cite{Latremoliere13b}. We also note that covariant tunnels are not constructed using a Lipschitz dynamical systems. They only involve an almost isometric isomorphism.
\end{remark}

The covariant propinquity is defined from certain quantities associated with covariant tunnels. These quantities do not depend on the quasi-Leibniz inequality. We now give their definitions, as they will be helpful with our current work.

\begin{notation}
  Let $\pi : \A \rightarrow \B$ be a positive unital linear map between two unital C*-algebras $\A$ and $\B$. We denote the dual map $\varphi \in \StateSpace(\B) \mapsto \varphi\circ\pi \in \StateSpace(\A)$ by $\pi^\ast$.
\end{notation}

\begin{notation}
  If $(E,d)$ is a metric space, then the Hausdorff distance \cite{Hausdorff} defined on the space of the closed subsets of $(E,d)$ is denoted by $\Haus{d}$. In case $E$ is a normed vector space and $d$ is the distance associated with some norm $N$, we write $\Haus{N}$ for $\Haus{d}$.
\end{notation}

\begin{definition}[{\cite[Definition 2.11]{Latremoliere14}}]\label{extent-def}
  Let $\mathds{A}_1 = (\A_1,\Lip_1,G_1,\delta_1,\alpha_1)$ and $\mathds{A}_2 = (\A_2,\Lip_2,\allowbreak G_2,\delta_2,\alpha_2)$ be two Lipschitz dynamical systems. The \emph{extent $\tunnelextent{\tau}$} of a covariant tunnel $\tau = (\D,\Lip_\D,\pi_1,\pi_2,\varsigma_1,\varsigma_2)$ from $\mathds{A}_1$ to $\mathds{A}_2$ is given as:
\begin{equation*}
\max\left\{ \Haus{\Kantorovich{\Lip}}\left(\StateSpace(\D),\pi_j^\ast\left(\StateSpace(\A_j)\right)\right) \middle\vert j\in\{1,2\} \right\} \text{.}
\end{equation*}
\end{definition}

\begin{definition}[{\cite{Latremoliere18b}}]\label{reach-def}
  Let $\varepsilon > 0$. Let $\mathds{A}_1 = (\A_1,\Lip_1,G_1,\delta_1,\alpha_1)$ and $\mathds{A}_2 = (\A_2,\Lip_2,G_2,\delta_2,\allowbreak \alpha_2)$ be two Lipschitz dynamical systems. The \emph{$\varepsilon$-reach $\tunnelreach{\tau}{\varepsilon}$} of a $\varepsilon$-covariant tunnel $\tau = (\D,\Lip_\D,\pi_1,\pi_2,\varsigma_1,\varsigma_2)$ from $\mathds{A}_1$ to $\mathds{A}_2$ is given as:
\begin{equation*}
\max_{\{j,k\}=\{1,2\}}\sup_{\varphi\in\StateSpace(\A_j)} \inf_{\psi\in\StateSpace(\A_k)}\sup_{g \in G_j\left[\frac{1}{\varepsilon}\right]} \Kantorovich{\Lip_\D}(\varphi\circ\alpha_j^g\circ\pi_j, \psi\circ\alpha_k^{\varsigma_j(g)}\circ\pi_k)
\end{equation*}
\end{definition}

The magnitude of a covariant tunnel summarizes all the data computed above.

\begin{definition}[{\cite{Latremoliere18b}}]\label{magnitude-def}
  Let $\varepsilon > 0$. The \emph{$\varepsilon$-magnitude} $\tunnelmagnitude{\tau}{\varepsilon}$ of a $\varepsilon$-covariant tunnel $\tau$ is the maximum of its $\varepsilon$-reach and its extent:
  \begin{equation*}
    \tunnelmagnitude{\tau}{\varepsilon} = \max\left\{ \tunnelreach{\tau}{\varepsilon}, \tunnelextent{\tau} \right\} \text{.}
  \end{equation*}
\end{definition}

We then define the covariant propinquity between Lipschitz dynamical systems as follows. While there are many appropriate choices for a class of tunnel used in the following definition as discussed in \cite{Latremoliere18b}, we will focus on the class of \emph{all} covariant $F$-tunnels. Thus, for a permissible function $F$ and for any $\varepsilon > 0$, and for any two Lipschitz dynamical systems $\mathds{A}$ and $\mathds{B}$, we denote the class of all $\varepsilon$-covariant $F$-tunnels from $\mathds{A}$ to $\mathds{B}$ by:
\begin{equation*}
  \tunnelset{\mathds{A}}{\mathds{B}}{F}{\varepsilon} \text{.}
\end{equation*}

\begin{definition}[{\cite{Latremoliere18b}}]\label{covariant-propinquity-def}
  Let $F$ be a permissible function. For $\mathds{A},\mathds{B}$ two Lipschitz dynamical $F$-system, the \emph{covariant $F$-propinquity} $\covpropinquity{F}(\mathds{A},\mathds{B})$ is defined as:
\begin{equation*}
  \min\left\{ \frac{\sqrt{2}}{2}, \inf\left\{ \varepsilon > 0 \middle\vert \exists \tau \in \tunnelset{\mathds{A}}{\mathds{B}}{F}{\varepsilon} \quad \tunnelmagnitude{\tau}{\varepsilon} \leq \varepsilon \right\} \right\} \text{.}
\end{equation*}
\end{definition}

In \cite{Latremoliere18b}, we prove that $\covpropinquity{F}$ is indeed a metric up to equivariant full quantum isometry.

\begin{theorem}[{\cite{Latremoliere18b}}]
  Let $F$ be a permissible function. If $(\A,\Lip_\A,G,\delta_G,\alpha)$ and $(\B,\Lip_\B,H,\delta_H,\allowbreak \beta)$ in are two Lipschitz dynamical $F$-systems, then:
   \begin{equation*}
     \covpropinquity{F}((\A,\Lip_\A,G,\delta_G,\alpha),(\B,\Lip_\B,H,\delta_H,\beta)) = 0
   \end{equation*}
if and only if there exists a full quantum isometry $\pi : (\A,\Lip_\A)\rightarrow(\B,\Lip_\B)$ and an isometric isomorphism of monoids $\varsigma: G\rightarrow H$ such that:
   \begin{equation*}
     \forall g \in G \quad \varphi\circ\alpha^g = \beta^{\varsigma(g)}\circ\varphi \text{.}
   \end{equation*}
i.e. $(\A,\Lip_\A,G,\delta_G,\alpha)$ and $(\B,\Lip_\B,H,\delta_H,\beta)$ are isomorphic as Lipschitz dynamical systems.

Moreover, $\covpropinquity{F}$ satisfies the triangle inequality and is symmetric in its arguments, so it defines a metric on the class of Lipschitz dynamical $F$-systems up to equivariant full quantum isometries.
\end{theorem}

This paper is concerned with the question of the completeness of the covariant propinquity on certain classes of Lipschitz dynamical systems. We begin with the matter of completeness for $\Upsilon$.

\section{Cauchy sequences of proper monoids for $\Upsilon$}

An interesting problem arises when studying the completeness of $\Upsilon$: given a Cauchy sequence for $\Upsilon$, the construction of a potential limit guided by the completeness of the Gromov-Hausdorff distance may not be a topological monoid in general, without assuming some form of uniform equicontinuity of the right translations of the monoids in our sequence --- properly defined, as we shall see below. This is actually a common condition to impose on functions over sequences of metric spaces converging for the Gromov-Hausdorff distance in order to obtain a form of convergence of the functions themselves. This issue can be however managed for certain classes of proper monoids, as we will see at the end of this section.

We begin with a definition which we will use to capture the equicontinuity of right translations for a sequence of proper monoids.

\begin{notation}
  We write:
  \begin{equation*}
    \prod_{n\in\N} G_n = \left\{ (g_n)_{n\in\N} : \exists M > 0 \quad \forall n \in \N \quad g_n \in G_n[M] \right\} \text{.}
  \end{equation*}
\end{notation}

\begin{definition}\label{regular-sequence-def}
  Let $(G_n,\delta_n)_{n\in\N}$ be a sequence of proper monoids. The set of \emph{regular sequences} $\mathcal{R}((G_n,\delta_n)_{n\in\N})$ is:
      \begin{equation*}
        \left\{ (g_n)_{n\in\N} \in \prod_{n\in\N} G_n \middle\vert 
          \begin{array}{l}
            \forall\varepsilon > 0\quad\exists \omega>0 \quad \exists N \in \N \\
            \forall n \geq N \quad \forall h,k \in G_n \\
            \delta_n(h,k)<\omega \implies \delta_n(h g_n, k g_n) < \varepsilon \text{.}
          \end{array}  \right\} \text{.}
      \end{equation*}
\end{definition}

While it is unclear in general how large the set of regular sequences associated to a sequence of proper monoids may be, it is always a monoid.

\begin{lemma}\label{regular-lemma}
  $\mathcal{R}((G_n,\delta_n)_{n\in\N})$ is a monoid for the pointwise multiplication.
\end{lemma}
\begin{proof}
  First, we note that the sequence $(e_n)_{n\in\N}$ of the identity elements of $(G_n)_{n\in\N}$ is regular.

  Let $(g_n)_{n\in\N}$ and $(g'_n)_{n\in\N}$ be regular sequences. First note that since for all $n\in\N$, the metric $\delta_n$ is left invariant, we have for all $n\in\N$:
\begin{equation*}
  \delta_n(g_n g_n', e_n) \leq \delta_n(g_n g'_n, g_n) + \delta_n(g_n,e_n) = \delta_n(g'_n, e_n) + \delta_n(g_n,e_n)
\end{equation*}
and thus $(g_n g'_n)_{n\in\N}$ is bounded since $(g_n)_{n\in\N}$ and $(g'_n)_{n\in\N}$ are.

Let $\varepsilon > 0$. There exists $\omega>0$ and $N\in\N$ such that if $n\geq N$ and if $h,k \in G_n$, and if $\delta_n(h,k) < \omega$ then $\delta_n(h g'_n, k g'_n) < \varepsilon$. Now, there exists $\omega_2 > 0$ and $N_1 \in \N$ such that if $n\geq N_1$, if $h,k \in G_n$ and if $\delta_n(h,k) < \omega_2$ then $\delta_n(h g_n, k g_n) < \omega$. Hence if $n\geq\max\{N,N_1\}$ and if $h,k \in G_n$ and if $\delta_n(h,k) < \omega_2$ then:
\begin{align*}
  \delta_n(h g_n g'_n, k g_n g'_n) < \varepsilon \text{.}
\end{align*}
Hence $(g_n g'_n)_{n\in\N}$ is regular. Thus $\mathcal{R}$ is closed under pointwise product, hence it is a monoid, as the multiplication is easily checked to be associative.
\end{proof}

We now can prove our theorem on convergence of Cauchy sequences for our metric $\Upsilon$. If $(G_n,\delta_n)_{n\in\N}$ is a Cauchy sequence of proper monoids for $\Upsilon$, then there exists a subsequence $(G_{j(n)},\delta_{j(n)})_{n\in\N}$ such that:
\begin{equation*}
  \sum_{n=0}^\infty \Upsilon((G_{j(n)},\delta_{j(n)}),(G_{j(n+1)},\delta_{j(n+1)})) < \infty\text{.}
\end{equation*}
We will work with such subsequences in the next result.

\begin{theorem}\label{Upsilon-completeness-thm}
  Let $(G_n,\delta_n)_{n\in\N}$ be a sequence such that for all $n\in\N$, there exists $\varepsilon_n > 0$ and:
\begin{equation*}
(\varsigma_n,\varkappa_n) \in \UIso{\varepsilon_n}{(G_n,\delta_n)}{(G_{n+1},\delta_{n+1})}{\frac{1}{\varepsilon_n}}
\end{equation*}
such that:
  \begin{enumerate}
    \item $\sum_{n=0}^\infty \varepsilon_n < \infty$,
    \item for all $N\in\N$ and $g \in G_N\left[ \frac{1}{\sum_{n=N}^\infty \varepsilon_n} \right]$:
\begin{equation*}
  \varpi_N(g) = \left(\begin{cases}
      g_n = e_n \text{ if $n < N$,}\\
      g_n = g    \text{ if $n = N$,}\\
      g_n = \varsigma_{n-1}(g_{n-1}) \text{ if $n>N$.}
      \end{cases}\right)_{n\in\N} \in \mathcal{R}((G_n,\delta_n)_{n\in\N})\text{,}
  \end{equation*}
  \end{enumerate}
then there exists a proper monoid $(G,\delta)$ such that $\lim_{n\rightarrow\infty} \Upsilon((G_n,\delta_n),(G,\delta)) = 0$.
\end{theorem}

\begin{proof}
Without loss of generality, we can actually assume that $\sum_{j=0}^\infty\varepsilon_j < \frac{\sqrt{2}}{2}$ (by simply truncating our original sequence).

It will be helpful to define $\varsigma_n^k = \varsigma_{k-1} \circ \ldots \circ \varsigma_n$ and similarly $\varkappa_k^n = \varkappa_{n+1} \circ \ldots \circ \varkappa_k$ for $k > n \in \N$. We also set $\varsigma_n^k(g) = e_k$ and $\varkappa_k^n(h) = e_n$ for all $g\in G_n, h\in G_k$ and $k < n \in \N$ and $\varsigma_n^n$ and $\varkappa_n^n$ are set to the identity of $G_n$. By Lemma (\ref{uiso-composition-lemma}), we note that for all $k>n \in \N$:
\begin{equation*}
  (\varsigma_n^k,\varkappa_k^n) \in \UIso{\sum_{j=n}^{k-1} \varepsilon_j}{(G_n,\delta_n)}{(G_k,\delta_k)}{\frac{1}{\sum_{j=n}^{k-1}\varepsilon_j}} \text{.}
\end{equation*}

Let:
\begin{equation*}
  H_\infty = \left\{ (g_n)_{n\in\N} \in \mathcal{R}((G_n,\delta_n)_{n\in\N}) \middle\vert
    \begin{array}{l}
      \forall\varepsilon>0 \quad \exists N\in \N \\
      \forall n \geq N \quad \forall j > n \quad \delta_{j}(g_{j},\varsigma_n^j(g_n)) < \varepsilon
    \end{array}
 \right\}\text{.}
\end{equation*}

We first note that for all $N\in\N$ and $g \in G_N\left[ \frac{1}{\sum_{n=N}^\infty\varepsilon_n}\right]$, we have $\varpi_N(g) \in \mathcal{R}((G_j,\delta_j)_{j\in\N})$ by assumption. Moreover if $n\geq N$ and $j > N$, and if we write $\varpi_N(g)$ as $(g_k)_{k\in\N}$, then:
\begin{equation*}
  \delta_j(g_j,\varsigma_n^j(g_n)) = \delta_j(\varsigma_N^j(g),\varsigma_n^j\circ\varsigma_N^n(g)) = 0 \text{.}
\end{equation*}
Hence $\varpi_n(g) \in H_\infty$ and in particular, $H_\infty$ is not empty.

We also define the equivalence relation on $H_\infty$ by:
\begin{equation*}
  (g_n)_{n\in\N} \sim (h_n)_{n\in\N} \iff \lim_{n\rightarrow\infty} \delta_n(g_n,h_n) = 0\text{.}
\end{equation*}
We set $G_\infty = \bigslant{H_\infty}{\sim}$ and we set $q : H_\infty \twoheadrightarrow G_\infty$ the canonical surjection. 

We now define, for all $(g_n)_{n\in\N}, (h_n)_{n\in\N} \in \prod_{n\in\N} G_n$:
\begin{equation*}
D((g_n)_{n\in\N}, (h_n)_{n\in\N}) = \limsup_{n\rightarrow\infty}\delta_n(g_n, h_n) \text{.} 
\end{equation*}
The function $D$ is a pseudo-metric on $G_\infty$. 

We note that $D(g,h) = 0$ if and only if $g\sim h$. Consequently, $D$ induces a metric on $G_\infty$ which we denote as $\delta_\infty$.

We turn to the matter of defining a multiplication on $G_\infty$. First, we prove that $H_\infty$ is closed under pointwise multiplication. let $(g_n)_{n\in\N}, (g'_n)_{n\in\N} \in H_\infty$. By Lemma (\ref{regular-lemma}), the sequence $(g_n g_n')_{n\in\N}$ is regular. Moreover, since $(g_n)_{n\in\N}$, $(g'_n)_{n\in\N}$ and $(g_n g'_n)_{n\in\N}$ are bounded and $\lim_{n\rightarrow\infty}\varepsilon_n = 0$, there exists $N_0\in\N$ such that for all $n\geq N_0$, we have:
\begin{equation*}
  g_n, g'_n, g_n g'_n  \in G_n \left[\frac{1}{\sum_{j=N_0}^\infty \varepsilon_j}\right] \text{.}
\end{equation*}
Let $\varepsilon > 0$. Since $(g'_n)_{n\in\N}$ is regular, there exists $\omega>0$ and $N_1\in\N$ such that for all $n\geq N_1$, if $h,k\in G_n$, and if $\delta_n(h,k) < \omega$, then $\delta_n(h g'_n, k g'_n) < \frac{\varepsilon}{3}$. 

Let $N_2\in\N$ such that $\sum_{n=N_2}^\infty \varepsilon_n \leq \min\{\frac{\omega}{2},\frac{\varepsilon}{6}\}$. Let $N_3\in\N$ such that for all $n\geq N_3$ and for all $j > n$, we have $\delta_{j}(g'_{j},\varsigma_n^j(g'_n))<\frac{\varepsilon}{6}$. Let $N_4\in\N$ such that $\delta_{j}(g_{j},\varsigma_n^j(g_n))< \frac{\omega}{2}$ if $n\geq N_4$ and $j > n$. We note that if $n\geq \max\{N_0,N_2,N_3,N_4\}$ and $j > n$, by Assertion (1) of Lemma (\ref{almost-isoiso-lemma}):
\begin{align*}
  \delta_{n}(\varkappa_j^n(g_{j}),g_n) &\leq \delta_{j}(g_{j},\varsigma_n^j(g_n)) + \sum_{k=n}^j \varepsilon_k \leq \omega \text{ and }\delta_n(\varkappa_j^n(g'_{j}),g'_n) \leq \frac{\varepsilon}{3} \text{.}
\end{align*}

Using Definition (\ref{almost-iso-def}), if $n \geq \max\{N_0,N_1,N_2,N_3,N_4\}$ and $j > n$, then since $g_j, g_j' \in G_j\left[\frac{1}{\sum_{k=N_0}^\infty \varepsilon_k}\right]$, $g_n g_n' \in G_n\left[\frac{1}{\sum_{k=N_0}^\infty\varepsilon_k} \right]$, and $(\varsigma_n^j,\varkappa_n^j)$ is an $\frac{1}{\sum_{k=N_0}^\infty\varepsilon_k}$-local $\sum_{k=n}^j \varepsilon_k$-almost isometry from $G_n$ to $G_j$, we conclude:
\begin{align*}
  \delta_{j}(g_{j} g'_{j}, \varsigma_n^j(g_n g'_n)) 
  &\leq \sum_{k=n}^j \varepsilon_k + \delta_{n}(\varkappa_j^n(g_{j})\varkappa_j^n(g'_{j}),g_n g'_n) \\
  &\leq \frac{\varepsilon}{3} + \delta_{n}(\varkappa_n^j(g_{j})\varkappa_n^j(g'_{j}),\varkappa_j^n(g_{j})g'_{n}) + \delta_{n}(\varkappa_j^n(g_{j})g'_{n}, g_n g'_n) \\
  &\leq \frac{\varepsilon}{3} + \delta_n(\varkappa_j^n(g'_{j}),g'_{n}) + \frac{\varepsilon}{3} \\
  &\leq \varepsilon \text{.}
\end{align*}

Hence $(g_n g_n')_{n\in\N} \in H_\infty$.

Our next step is to prove that the pointwise product of equivalent sequences are again equivalent. Let $(g_n)_{n\in\N}$, $(g'_n)_{n\in\N}$, $(h_n)_{n\in\N}$ and $(h'_n)_{n\in\N}$ be four elements of $H_\infty$ such that $(g_n)_{n\in\N} \sim (g'_n)_{n\in\N}$ and $(h_n)_{n\in\N} \sim (h'_n)_{n\in\N}$. Let $\varepsilon > 0$. Since $(h'_n)_{n\in\N}$ is regular, there exists $\omega > 0$ and $N\in\N$ such that if $n\geq N$, if $j,k \in G_n$, and if  $\delta_n(k,j) < \omega$, then $\delta_n(k h'_n, j h'_n) < \frac{\varepsilon}{2}$.

Now, there exists $N_1 \in \N$ such that if $n\geq N_1$ then $\delta_n(g_n,g'_n) < \omega$. There exists $N_2 \in \N$ such that if $n\geq N_2$ then $\delta_n(h_n,h'_n) < \frac{\varepsilon}{2}$. Thus, we estimate that for $n\geq\max\{N,N_1,N_2\}$:
\begin{align*}
\delta_n(g_n h_n, g'_n h'_n) &\leq \delta_n(g_n h_n, g_n h'_n) + \delta_n(g_n h_n', g_n' h_n')\\
&\leq \delta_n(h_n, h'_n) + \frac{\varepsilon}{2} \leq \varepsilon \text{.}
\end{align*}

Hence $(g_n h_n)_{n\in\N} \sim (g_n' h'_n)_{n\in\N}$. We therefore define $gh$, for $g,h \in G_\infty$, to be the equivalence class of $(g_n h_n)_{n\in\N}$ for any $(g_n)_{n\in\N}, (h_n)_{n\in\N} \in H_\infty$ such that $q((g_n)_{n\in\N}) = g$ and $q((h_n)_{n\in\N}) = h$.

It is then easy to check that this operation is associative since the law is associative on each $G_n$ for all $n\in\N$. Moreover, it is easy to check that the equivalence class of $(e_n)_{n\in\N}$, where $e_n$ is the unit of $G_n$ for each $n\in\N$ is  the identity of $G_\infty$. 

Moreover, it is also immediate that the distance $\delta_\infty$ on $G_\infty$ is left-invariant for the multiplication thus defined since for all $n\in\N$, the distance $\delta_n$ is left-invariant.

We now turn to the continuity of the multiplication on $G_\infty$. Let $\varepsilon > 0$, $h = q((h_n)_{n\in\N}) \in G_\infty$, $g = q((g_n)_{n\in\N})$ and $h' = q((h'_n)_{n\in\N})$, such that $\delta_\infty(h,h') < \frac{\varepsilon}{2}$. By regularity, there exists $N\in\N$ and $\omega > 0$ such that for all $n\geq N$, if $g,h\in G_n$ and $\delta_n(g,h) < \omega$, then $\delta_n(g h_n,h h_n) < \varepsilon$. Let now $g' = q((g'_n)_{n\in\N})$ in $G_\infty$ such that $\delta_\infty(g,g') < \omega$. There exists $N_1 \geq N$ such that for all $n\geq N_1$, we have $\delta_n(g_n,g'_n) < \omega$ and $\delta_n(h_n,h_n')<\frac{\varepsilon}{2}$, and thus:
\begin{align*}
  \delta_n( g'_n h'_n, g_n h_n )  &\leq \delta_n(g'_n h'_n , g'_n h_n) + \delta_n(g'_n h_n, g_n h_n) \\
  &= \delta_n(h'_n,h_n) + \frac{\varepsilon}{2} \leq \varepsilon \text{.}
\end{align*}
So $\delta_\infty(g' h', g h) \leq \varepsilon$. Hence, the multiplication is jointly continuous at $(g,h) \in G_\infty^2$. In fact, $h \in G_\infty \mapsto g h$ is uniformly continuous for all $g\in G_\infty$.

We now check that the closed balls in $G_\infty$ for $\delta_\infty$ are totally bounded. Let $R > 0$. Let $\varepsilon > 0$. There exists $N\in\N$ such that $\sum_{n=N}^\infty\varepsilon_n < \frac{\varepsilon}{4}$ and $R + 1 + \frac{\varepsilon}{2} \leq \frac{1}{\sum_{n=N}^\infty\varepsilon_n} - \sum_{n=N}^\infty\varepsilon_n$. 

Let $g\in G_\infty[R]$. Let $(g_n)_{n\in\N} \in H_\infty$ such that $q((g_n)_{n\in\N}) = g$. We thus have $\limsup_{n\rightarrow\infty}\delta_n(g_n,e_n) < R + 1$. Thus there exists $N_1\in\N$, $N_1\geq N$ such that for all $n\geq N_1$, we have $\delta_n(g_n,e_n) \leq R + 1$. Note that:
\begin{equation*}
\forall n \geq N_1 \quad \delta_n(g_n,e_n) \leq R + 1 \leq \frac{1}{\sum_{k=N}^\infty\varepsilon_k} \leq \frac{1}{\sum_{k=N_1}^\infty \varepsilon_k} \leq \frac{1}{\sum_{k=n}^\infty\varepsilon_k} \text{.}
\end{equation*}
Of course, $N_1$ depends on $(g_n)_{n\in\N}$, a dependence which we now remove by changing our choice of a representative of $q(g)$.

Let us therefore define $g'_n \in G_n$ by setting $(g'_n)_{n\in\N}=\varpi_{N_1}(g_{N_1})$. By definition of $H_\infty$, we have:
\begin{align*}
  \lim_{n\rightarrow\infty} \delta_n(g'_{n}, g_{n}) &= \lim_{n\rightarrow\infty}\delta_n(\varsigma_{N_1}^n(g_{N_1}),g_n) = 0\text{.}
\end{align*}
Hence $(g'_n)_{n\in\N}\sim (g_n)_{n\in\N}$. On the other hand, $\delta_N(g'_N,e_N) \leq \frac{\varepsilon}{2} + \delta_{N_1}(g_{N_1},e_{N_1}) \leq R + 1 + \frac{\varepsilon}{2}$ by Assertion (4) of Lemma (\ref{almost-isoiso-lemma}).

Now, since $G_N$ is proper, the closed ball $G_N[R+1+\frac{\varepsilon}{2}]$ is compact, hence there exists a finite, $\frac{\varepsilon}{2}$-dense subset $F$ of this ball. We then note that there exists $h \in F$ such that $\delta_N(g_N,h) < \frac{\varepsilon}{2}$. Therefore by Assertion (4) of Lemma (\ref{almost-isoiso-lemma}):
\begin{align*}
  D((g'_n)_{n\in\N}, \varpi_N(h)) &= \limsup_{n\rightarrow\infty} \delta_n(\varsigma_N^n(g_N'), \varsigma_N^n(h)) \\
  &\leq \delta_N(g_N',h) + \frac{\varepsilon}{2} \leq \varepsilon \text{.}
\end{align*}
Hence $G_\infty[R]$ is totally bounded as desired.

It then follows that the metric completion $\overline{G_\infty}$ of $G_\infty$ for $\delta_\infty$ is a proper metric space: if $R > 0$ then $\overline{G_\infty}[R]$ lies inside the closure of $G_\infty[R+1]$ and thus it is totally bounded. As a totally bounded, closed subset of a complete metric space, we conclude $\overline{G_\infty}[R]$ is compact as desired.

Moreover, for all $g\in G_\infty$, the map $h \in G_\infty \mapsto g h$ is uniformly continuous and thus, it admits a unique extension to $\overline{G_\infty}$. We note that by continuity, for any $h, h' \in \overline{G_\infty}$, and $g \in G_\infty$, we have:
\begin{align*}
  \left|\delta_\infty(g h, g h') - \delta_\infty(h,h')\right| 
  &= \lim_{\substack{ k \rightarrow h \\ k \in G_\infty }} \lim_{{\substack{ k' \rightarrow h' \\ k' \in G_\infty }}} \left| \delta_\infty(g k, g k') - \delta_\infty(k,k')\right| \\
  &= \lim_{\substack{ k \rightarrow h \\ k \in G_\infty }} \lim_{{\substack{ k' \rightarrow h' \\ k' \in G_\infty }}} \left| \delta_\infty(k,k') - \delta_\infty(k,k') \right|  = 0 \text{.} 
\end{align*}

Now, let $h \in \overline{G_\infty}$. Let $\varepsilon > 0$. There exists $h' \in G_\infty$ with $\delta_\infty(h,h') < \frac{\varepsilon}{3}$. Moreover, there exists $\omega > 0$ and $N\in\N$ such that if $n\geq N$, $g,g' \in G_n$ and $\delta_n(g,g') < \omega$ then $\delta_n(g h'_n, g' h'_n) < \frac{\varepsilon}{3}$. Hence, for all $g,g' \in G_\infty$ with $\delta_\infty(g,g') < \omega$, we have:
\begin{align*}
  \delta_\infty(g h, g' h) &\leq \delta_\infty(g h, g h') + \delta_\infty(g h', g' h') + \delta_\infty(g' h', g' h) \\
  &= \delta_\infty(h,h') + \limsup_{n\rightarrow\infty} \delta_n (g_n h'_n, g'_n h'_n) + \delta_\infty(h',h) \\
  &\leq \frac{\varepsilon}{3} + \frac{\varepsilon}{3} + \frac{\varepsilon}{3} = \varepsilon \text{.}
\end{align*}
Thus $g \in G_\infty \mapsto g h$ is uniformly continuous for any $h \in \overline{G_\infty}$, and thus it too admits a unique extension to $\overline{G_\infty}$. We have defined a multiplication on $\overline{G_\infty}$. 

Now, for all $h,h' \in \overline{G_\infty}$ and for all $g \in \overline{G_\infty}$, using continuity, we obtain for any $g' \in G_\infty$: 
\begin{align*}
\left|\delta_\infty(g h, g h') - \delta_\infty(h,h')\right|&\leq\left|\delta_\infty(g h, g h') - \delta_\infty(g' h, g h')\right| + \left|\delta_\infty(g' h, g h') - \delta_\infty(h,h')\right| \\
&\leq \delta_\infty(g h, g' h) + \left|\delta_\infty(g' h, g h') - \delta_\infty(g' h, g' h')\right| \\
&\quad +\left|\delta_\infty(g' h, g' h') - \delta_\infty(h,h')\right|\\
&\leq \delta_\infty(g h, g' h) + \delta_\infty(g h', g' h') + 0 \\
&\xrightarrow{\substack{g' \rightarrow g \\ g'\in G_\infty}} 0 \text{.}
\end{align*}
Hence $\delta_\infty$ is left invariant by our newly defined multiplication on $\overline{G_\infty}$.

Furthermore, let $h \in \overline{G_\infty}$. Let $\varepsilon > 0$. By uniform continuity of $g\in \overline{G_\infty}\mapsto g h$, there exists $\omega>0$ such that if $g,g' \in \overline{G_\infty}$ and $\delta_\infty(g,g') < \omega$ then $\delta_\infty(g h, g' h) \leq \frac{\varepsilon}{2}$. Consequently, if $g', h' \in \overline{G_\infty}$ with $\delta_\infty(g,g') < \omega$ and if $\delta_\infty(h,h') < \frac{\varepsilon}{2}$ then:
\begin{equation*}
  \delta_\infty(g' h', g h) \leq \delta_\infty(g' h', g' h) + \delta_\infty(g' h, g h) = \delta_\infty(h,h') + \frac{\varepsilon}{2} < \varepsilon \text{.}
\end{equation*}
Therefore, our multiplication is indeed jointly continuous at every point of $G_\infty$ and $\delta_\infty$ is left-invariant for the multiplication. Therefore, $\overline{G_\infty}$ is a proper monoid, as desired.

Our last step is to prove that $(G_n,\delta_n)_{n\in\N}$ converges to $(\overline{G_\infty},\delta_\infty)$ for $\Upsilon$. To begin with, we denote $q\circ\varpi_N$ as $\varpi_N$ for all $N\in\N$, to keep our notations simple. We now define the other maps for our almost isometric isomorphisms.

Let $N\in\N$. For any $g \in \overline{G_\infty}$, there exists $\psi_N(g) \in H_\infty$ such that $\delta_\infty(g,q\circ\psi_N(g)) < \varepsilon_n$. Writing $\psi_N(g) = (h_n)_{n\in\N}$, we then set $\sigma_N(g) = h_N$. Of course, this definition depends on our choice function $\psi_N$.

Let $\varepsilon \in \left(0 , \frac{1}{2} \right)$, and let $N\in\N$ such that $\sum_{n=N}^\infty\varepsilon_n < \frac{\varepsilon}{4}$. Note that $\frac{1}{\sum_{n=N}^\infty\varepsilon_n} > \frac{4}{\varepsilon} > \frac{1}{\varepsilon}$. So for all $n>N$, we note that:
\begin{equation*}
  (\varsigma_N^n,\varkappa_n^N) \in \UIso{\sum_{j=N}^\infty\varepsilon_j}{(G_N,\delta_N)}{(G_n,\delta_n)}{\frac{4}{\varepsilon}} \text{.}
\end{equation*}

Let $n\geq N$, $g,h \in G_N\left[\frac{1}{\varepsilon}\right]$, and $k \in \overline{G_\infty}\left[\frac{1}{\varepsilon}\right]$. We write $(k_n)_{n\in\N} = \psi_N(k)$. Note that if $e = q((e_n)_{n\in\N})$ then:
\begin{align*}
  D(\psi_N(k),(e_n)_{n\in\N}) 
  &= \delta_\infty(q\circ\psi_N(k), e) \leq \delta_\infty(q\circ\psi_N(k),k) + \delta_\infty(e, k)\\
  &\leq \frac{1}{\varepsilon} + \varepsilon_N < \frac{1}{\varepsilon} + \frac{\varepsilon}{4} \leq \frac{2}{\varepsilon} \text{.} 
\end{align*}
By definition of $H_\infty$, there exists $N_1\in\N$ such that for all $n\geq N_1$ and for all $j > n$, we have $\delta_j(k_j,\varsigma_n^j(k_n)) < \frac{\varepsilon}{4}$. Now for all $n\geq \max\{N, N_1\}$ and $j >  n$, we note that:
\begin{equation*}
  \delta_n(k_n,\varkappa_j^n(k_j)) \leq \delta_j(k_j,\varsigma_n^j(k_n)) + \sum_{k=n}^{j-1}\varepsilon_k \leq \frac{\varepsilon}{2}\text{.}
\end{equation*}

We then estimate, for all $n\geq \max\{N,N_1\}$:
\begin{multline*}
  \left| \delta_\infty(\varpi_n(g)\varpi_n(h),k) - \delta_n(g h, \sigma_n(k))\right|\\
  \begin{split}
    &\leq  \left| \delta_\infty(\varpi_n(g)\varpi_n(h),\psi_N(k)) - \delta_n(g h, \sigma_n(k))\right| + \varepsilon_n \\
    &= \left| D(\varpi_n(g)\varpi_n(h),(k_j)_{j\in\N}) - \delta_n(g h, \sigma_n((k_j)_{j\in\N})) \right| + \varepsilon_n \\
    &\leq \limsup_{j\rightarrow\infty} \left|\delta_j(\varsigma_n^j(g)\varsigma_n^j(h),k_j) - \delta_n(g h, k_n)\right| + \varepsilon_n \\
    &\leq \limsup_{j\rightarrow\infty} \left| \delta_j(\varsigma_n^j(g)\varsigma_n^j(h),k_j) - \delta_n(g h, \varkappa_j^n(k_j)) \right| \\
    &\quad + \limsup_{j\rightarrow\infty} \left|\delta_n(g h, \varkappa_j^n(k_j)) - \delta_n(g h, k_n)\right| + \varepsilon_n \\
    &\leq \limsup_{j\rightarrow\infty} \left| \delta_j(\varsigma_n^j(g)\varsigma_n^j(h),k_j) - \delta_n(g h, \varkappa_j^n(k_j)) \right| \\
    &\quad + \limsup_{j\rightarrow\infty} \delta_n(k_n, \varkappa_j^n(k_j)) + \varepsilon_n \\
    &\leq \sum_{j=n}^\infty \varepsilon_j + \limsup_{j\rightarrow\infty} \delta_n(k_n, \varkappa_j^n(k_j)) + \varepsilon_n \\
    &\leq \frac{\varepsilon}{4} + \frac{\varepsilon}{2} + \frac{\varepsilon}{4}  = \varepsilon \text{.}
  \end{split}
\end{multline*}

Now, let $g,h \in \overline{G_\infty}\left[\frac{1}{\varepsilon}\right]$ and write $(g_n)_{n\in\N} = \psi_N(g)$ and $(h_n)_{n\in\N} = \psi_N(h)$. If $k \in G_N\left[\frac{1}{\varepsilon}\right]$ then:
\begin{multline*}
  \left| \delta_N(\sigma_N(g)\sigma_N(h), k)  - \delta_\infty(g h, \varpi_N(k)) \right| \\
  \begin{split}
    &\leq \limsup_{n\rightarrow\infty} \left| \delta_N(g_N h_N, k)  - \delta_n(g_n h_n, \varsigma_N^n(k)) \right|\\
    &\leq  \limsup_{n\rightarrow\infty} \left| \delta_N(g_N h_N, k)  - \delta_N(\varkappa_n^N(g_n)\varkappa_n^N(h_n), k) \right|\\
    &\quad + \limsup_{n\rightarrow\infty} \left| \delta_N(\varkappa_n^N(g_n)\varkappa_n^N(h_N), k) - \delta_n(g_n h_n, \varsigma_N^n(k)) \right| \\
    &\leq   \limsup_{n\rightarrow\infty} \left| \delta_N(g_N h_N, \varkappa_n^N(g_n)\varkappa_n^N(h_n)) \right|\\
    &\quad + \limsup_{n\rightarrow\infty} \left| \delta_N(\varkappa_n^N(g_n)\varkappa_n^N(h_n), k) - \delta_n(g_n h_n, \varsigma_N^n(k)) \right| \\
    &\leq   \limsup_{n\rightarrow\infty} \left| \delta_n(\varsigma_N^n(g_N h_N), g_n h_n)) \right| + 2\sum_{j=N}^n\varepsilon_j \leq \varepsilon \text{.}\\
  \end{split}
\end{multline*}

Thus $(\varpi_N,\sigma_N) \in \UIso{\varepsilon}{(G_N,\delta_N)}{(\overline{G_\infty},\delta_\infty)}{\frac{1}{\varepsilon}}$. In particular, if $n\geq N$ then:
\begin{equation*}
  \Upsilon((G_N,\delta_N),(\overline{G_\infty},\delta_\infty)) \leq \varepsilon \text{,}
\end{equation*}
and our proof is concluded.
 \end{proof}

We emphasize that a priori, a Cauchy sequence of proper groups for $\Upsilon$ which meets the assumptions of Theorem (\ref{Upsilon-completeness-thm}) will indeed converge to a proper monoid, but maybe not to a group. In order to assure that the limit is indeed a group, a new assumption must be added to Theorem (\ref{Upsilon-completeness-thm}). In turn, this new assumption implies regularity for all sequences. We begin with relating inverse and local almost isometric isomorphism.
\begin{lemma}\label{uiso-inverse-lemma}
  Let $(G,\delta_G)$ and $(H,\delta_H)$ be two metric groups. Let $\varepsilon > 0$. If:
  \begin{equation*}
    (\varsigma,\varkappa) \in \UIso{\varepsilon}{(G,\delta_G)}{(H,\delta_H)}{\frac{1}{\varepsilon}}
  \end{equation*}
  then for all $g \in G\left[\frac{1}{\varepsilon}\right]$ such that $g^{-1} \in G\left[\frac{1}{\varepsilon}\right]$, the following estimate holds:
  \begin{equation*}
    \delta_H\left(\varsigma(g)^{-1},\varsigma\left(g^{-1}\right)\right) \leq \varepsilon \text{.}
  \end{equation*}
\end{lemma}

\begin{proof}
  We compute:
  \begin{align*}
    \delta_H\left(\varsigma(g)^{-1},\varsigma\left(g^{-1}\right)\right) 
    &= \delta_H\left(\varsigma(g) \varsigma(g)^{-1}, \varsigma(g)\varsigma\left(g^{-1}\right)\right)  \\
    &= \delta_H\left(e_H, \varsigma(g)\varsigma\left(g^{-1}\right)\right)\\
    &\leq \varepsilon + \delta_G(e_G,g g^{-1}) = \varepsilon \text{,}
  \end{align*}
  as desired.
\end{proof}

We thus get our result for convergence of Cauchy sequence of proper groups for $\Upsilon$.
\begin{corollary}\label{Upsilon-completeness-cor}
  If $(G_n,\delta_n)_{n\in\N}$ is a Cauchy sequence of proper \emph{groups} for $\Upsilon$ such that for all $\varepsilon > 0$ there exists $\omega > 0$ and $N\in \N$ such that for all $n\geq N$:
\begin{equation*}
  \forall g,h \in G_n \quad \delta_n(g,h) < \omega \implies \delta_n\left(g^{-1},h^{-1}\right) < \varepsilon\text{,}
\end{equation*}
then there exists a proper \emph{group} $(G,\delta)$ such that $\lim_{n\rightarrow\infty} \Upsilon((G_n,\delta_n),(G,\delta)) = 0$.
\end{corollary}

\begin{proof}
First, let $(g_n)_{n\in\N} \in \prod_{n\in\N} G_n$ and let $\varepsilon > 0$. By our assumption, there exists $N\in\N$ and $\omega > 0$ such that if $n\geq N$, if $h,k \in G_n$ and $\delta_n(h,k) < \omega$ then $\delta_n(h^{-1},k^{-1}) < \varepsilon$. Now, there exists $N_1\in\N$ and $\eta > 0$ such that for all $n\geq N_1$ and if $h,k \in G_n$ with $\delta_n(h,k) < \eta$ then $\delta_n(h^{-1},k^{-1}) < \omega$.

Thus for $n\geq \max\{ N, N_1 \}$ and for all $h,k \in G_n$ with $\delta_n(h,k) < \eta$, we conclude $ \delta_n(g_n^{-1} h^{-1}, g_n^{-1} k^{-1}) = \delta_n(h^{-1},k^{-1}) < \omega$. Hence:
\begin{equation*}
  \delta_n(h g_n, k g_n) = \delta_n( (g_n^{-1} h^{-1})^{-1}, (g_n^{-1} k^{-1})^{-1} ) < \varepsilon \text{.}
\end{equation*} 
So $(g_n^{-1})_{n\in\N}$ is regular. We thus conclude that $\mathcal{R}((G_n,\delta_n)_{n\in\N}) = \prod_{n\in\N} G_n$.

Since $(G_n,\delta_n)_{n\in\N}$ is Cauchy, up to extracting a subsequence, we can choose $(\varepsilon_n)_{n\in\N}$ such that $\sum_{n=0}^\infty \varepsilon_n < \infty$ and $\Upsilon((G_n,\delta_n),(G_{n+1},\delta_{n+1})) < \varepsilon_n$. Let:
\begin{equation*}
  (\varsigma_n,\varkappa_n) \in \UIso{\varepsilon_n}{(G_n,\delta_n)}{(G_{n+1},\delta_{n+1})}{\frac{1}{\varepsilon_n}} \text{.}
\end{equation*}
We will use the notations and observations of the proof of Theorem (\ref{Upsilon-completeness-thm}).

Let now $(g_n)_{n\in\N} \in H_\infty$. First, note that the left invariance of the metric $\delta_n$ for all $n\in\N$, the sequence $(g_n^{-1})_{n\in\N}$ is bounded. Let $\varepsilon > 0$. There exists $N\in\N_1$ and $\omega > 0$ such that if $n\geq N_1$ and $h,k \in G_n$ with $\delta_n(h,k) < \omega$ then $\delta_n(h^{-1},k^{-1}) < \frac{\varepsilon}{2}$. Now, there exists $N_2 \in \N$ such that if $n\geq N_2$ and $j > n$ then $\delta_j(g_j,\varsigma_n^j(g_n)) < \omega$, and thus $\delta_j(g_j^{-1}, \varsigma_n^{j}(g_n)^{-1}) < \frac{\varepsilon}{2}$. There exists $N_3 \in \N$ such that $\sum_{n=N_3}^\infty\varepsilon_n < \frac{\varepsilon}{2}$ and $g_n, g_n^{-1} \in G_n\left[ \frac{1}{\sum_{j=N_3}^\infty \varepsilon_j} \right]$. Thus, for all $n\geq\max\{N_1,N_2,N_3\}$ and $j > n$, we have, using Lemma (\ref{uiso-inverse-lemma}):
\begin{equation*}
  \delta_j(g_j^{-1},\varsigma_n^j(g_n^{-1})) \leq \delta_j(g_j^{-1},\varsigma_n^j(g_n)^{-1}) + \delta_j(\varsigma_n^j(g_n)^{-1},\varsigma_n^j(g_n^{-1}))  < \varepsilon\text{.}
\end{equation*}
Hence, $(g_n^{-1})_{n\in\N} \in H_\infty$. It is now sufficient to observe, using the notations of the proof of Theorem (\ref{Upsilon-completeness-thm}), that for any $(g_n)_{n\in\N}$ and $(h_n)_{n\in\N}$ chosen in $H_\infty$ and $(g_n)_{n\in\N}\sim (h_n)_{n\in\N}$ then, for all $\varepsilon > 0$, there exists $\omega > 0$ and $N\in\N$ such that if $n\geq N$ and $\delta_n(g,h) < \omega$ for any $g,h \in G_n$ then $\delta_n(g^{-1},h^{-1})<\varepsilon$; since there exists $N_1\in\N$ such that $\delta_n(g_n,h_n) < \omega)$ for all $n\geq N_1$, we conclude that $\delta_n(g_n^{-1},h_n^{-1}) < \varepsilon$. Hence $(g_n^{-1})_{n\in\N}$ and $(h_n^{-1})_{n\in\N}$ are equivalent and thus, we can define the inverse map on $G_\infty$ by setting $g^{-1} = q((g_n^{-1})_{n\in\N})$ for any $(g_n)_{n\in\N} \in q^{-1}(g)$. It is immediate that this map is the inverse on the monoid $G_\infty$ for the multiplication of $G_\infty$, turning $G_\infty$ into a group.

Now, let $\varepsilon > 0$, and let $N\in\N$ and $\omega > 0$ given as above. If $\delta_\infty(g,h) < \omega$ for some $g=q((g_n)_{n\in\N})$ and $h = q((h_n)_{n\in\N})$ then there exists $N_1 \geq N$ such that $\delta_n(g_n,h_n) < \omega$ and therefore, $\delta_n(g_n^{-1},h_n^{-1}) < \varepsilon$, so $\delta_\infty(g^{-1},h^{-1})<\varepsilon$. Thus the inverse map is uniformly continuous on $G_\infty$ and can be extended to $\overline{G_\infty}$ on which it is now easy to check, it is the inverse for the multiplication of $\overline{G_\infty}$, hence turning $\overline{G_\infty}$ into a topological group.
\end{proof}

Of course, the multiplicative group $(0,\infty)$ does not have a uniformly continuous inverse, so the assumption of Corollary (\ref{Upsilon-completeness-cor}) is strong, though  not unreasonable, and it is useful in controlling the regularity condition of Theorem (\ref{Upsilon-completeness-thm}).  We now discuss some other natural conditions under which the regularity condition is controllable. The easiest situation is given as follows.

\begin{corollary}
The metric $\Upsilon$ restricted to the class of proper monoids with bi-invariant metric is complete. Moreover, $\Upsilon$ restricted to the class  of proper groups with bi-invariant metric is also complete.
\end{corollary}

\begin{proof}
Let $(G_n,\delta_n)_{n\in\\N}$ be a Cauchy sequence for $\Upsilon$ such that for all $n\in\N$, the metric $\delta_n$ is bi-invariant. There exists a subsequence $(G_{j(n)},\delta_{j(n)})_{n\in\N}$ of $(G_n,\delta_n)_{n\in\N}$ such that $\sum_{n=0}^\infty\Upsilon((G_{j(n)},\delta_{j(n)}),(G_{j(n+1)},\delta_{j(n+1)})) < \infty$. We immediately check that $\mathcal{R}((G_{j(n)},\delta_{j(n)})_{n\in\N}) = \prod_{n\in\N} G_{j(n)}$, so we can apply Theorem (\ref{Upsilon-completeness-thm}) to conclude that $(G_{j(n)},\delta_{j(n)})_{n\in\N}$ converges for $\Upsilon$, and thus, as a Cauchy sequence with a convergent subsequence, so does $(G_n,\delta_n)_{n\in\N}$. This proves that $\Upsilon$ restricts to a complete metric on the class of proper monoids with bi-invariant metrics.

Note last that if $(G,\delta)$ is a proper group with $\delta$ bi-invariant, then for all $g,h \in G$ we have $\delta(g^{-1},h^{-1}) = \delta(g g^{-1}, g h^{-1}) = \delta(h, g h^{-1} h) = \delta(h,g)$ so the inverse map is an isometry, hence Corollary (\ref{Upsilon-completeness-cor}) applies.
\end{proof}

Another situation where the regularity condition in Theorem (\ref{Upsilon-completeness-thm}) can be handled, in principle, is when the right translations are Lipschitz. We just need to control the Lipschitz constant, rather than a whole modulus of continuity, so we can define the following:
\begin{equation*}
\Upsilon_\ast(G,H) = \inf\left\{\varepsilon>0\middle\vert\begin{array}{l}
                                                           \exists (\varsigma,\varkappa)\in \UIso{\varepsilon}{G}{H}{\frac{1}{\varepsilon}} \\
\sup_{g\in G} \left|\dil{h\in H\mapsto hg}-\dil{h\in H\mapsto h\varsigma(g)}\right| < \varepsilon\\
\sup_{h\in H} \left|\dil{g\in G\mapsto gh}-\dil{g\in G\mapsto g\varkappa(h)}\right| < \varepsilon
\end{array}\right\}
\end{equation*}
where $\dil{f}$ is meant as the best Lipschitz constant for a function $f$ between metric spaces. Now, convergence for $\Upsilon_\ast$ implies in particular that we can find almost isometric isomorphism which will meet our regularity condition in Theorem (\ref{Upsilon-completeness-thm}).

\section{Cauchy Sequences for the Covariant Propinquity}

We now study the problem of convergence of Cauchy sequences for the covariant propinquity. We begin with the following corollary of \cite[Theorem 2.13]{Latremoliere17c}, which extends \cite[Theorem 3.10]{Latremoliere17c} to the proper setting we are now working within. This result encapsulates some of the covariance property of the propinquity itself. We will use our work in \cite{Latremoliere13b} and \cite{Latremoliere17c} and in particular, we recall what a target set and a forward target set is.

Let $(\A,\Lip_\A)$ and $(\B,\Lip_\B)$ be two {\qcms s}. Let $\tau = (\D,\Lip_\D,\pi_\A,\pi_\B)$ be a tunnel from $(\A,\Lip_\A)$ to $(\B,\Lip_\B)$. For any $a\in \dom{\Lip_\A}$ and $l\geq\Lip_\A(a)$, the \emph{$l$-target set of $a$} is defined by:
\begin{equation*}
  \targetsettunnel{\tau}{a}{l} = \left\{ \pi_\B(d) \middle\vert d\in \sa{\D}, \Lip_\D(d) \leq l, \pi_\A(d) = a \right\} \text{.} 
\end{equation*}
Now, if $\tau = (\D,\Lip_\D,\pi_\A,\pi_\B,\varsigma,\varkappa)$ is a covariant tunnel, then for all $a\in\sa{\A}$ and $l \geq \Lip_\A(a)$, by a mild abuse of notations, we write $\targetsettunnel{\tau}{a}{l}$ for $\targetsettunnel{\tau'}{a}{l}$ where $\tau' = (\D,\Lip_\D,\pi_\A,\pi_\B)$.

Moreover, we denote $(\D,\Lip_\D,\pi_\B,\pi_\A,\varkappa,\varsigma)$ as $\tau^{-1}$. 

Now, by \cite[Corollary 4.5]{Latremoliere13b},\cite[Proposition 2.12]{Latremoliere14}, if $a, a' \in \dom{\Lip_\A}$ and $l\geq \max\{\Lip_\A(a),\Lip_\A(a')\}$, and if $b\in\targetsettunnel{\tau}{a}{l}$ and $b' \in \targetsettunnel{\tau}{a'}{l}$ then:
\begin{equation*}
  \norm{b - b'}{\B} \leq \norm{a - a'}{\A} + 2 l \tunnelextent{\tau} \text{.}
\end{equation*}

Let now $\varphi : \B \rightarrow \B$ be a Lipschitz linear map. Let $\tau$ be a tunnel from $(\A,\Lip_\A)$ to $(\B,\Lip_\B)$. For $a \in \dom{\Lip_\A}$ and $l \geq \Lip_\A(a)$, the \emph{$l$-image-target set} of $a$ is defined by:
\begin{equation*}
\targetsetimage{\tau,\varphi}{a}{l} = \varphi\left(\targetsettunnel{\tau}{a}{l}\right)
\end{equation*}
and the \emph{$(l,D)$-forward-target set} of $a$, for $D\geq\max\{ 1, \dil{\varphi} \}$ is defined by:
\begin{equation*}
\targetsetforward{\tau,\varphi}{a}{l,D} = \targetsettunnel{\tau^{-1}}{\varphi\left(\targetsettunnel{\tau}{a}{l}\right)}{D l} \text{.}
\end{equation*}

Now, by \cite[Lemma 2.5]{Latremoliere17c}, if $a,a' \in \sa{\A}$ and $l\geq\max\{\Lip_\A(a),\Lip_\A(a')\}$, and if $f \in \targetsetforward{\tau,\varphi}{a}{l,D}$ and $f' \in \targetsetforward{\tau}{a'}{l,D}$ then:
\begin{equation*}
  \norm{f-f'}{\A} \leq D \left( \norm{a-a'}{\A} + 8 l \tunnelextent{\tau} \right) \text{.}
\end{equation*}
As before, if $\tau$ is a covariant tunnel, we write $\targetsetforward{\tau,\varphi}{\cdot}{\cdot}$ for the forward target set associated to the underlying tunnel of $\tau$.

We now recall and mildly extend a metric introduced in \cite{Latremoliere16b}.

\begin{theorem}[{\cite{Latremoliere16b}}]
  Let $(\A,\Lip_\A)$ be a {\qcms} and let $\B$ be a unital C*-algebra. If for any two unital linear maps $\alpha$, $\beta$ from $\A$ to $\B$, we set:
  \begin{equation*}
    \KantorovichDist{\Lip_\A}{\alpha}{\beta} = \sup\left\{ \norm{\alpha(a) - \beta(a)}{\B} : a\in\dom{\Lip}, \Lip(a) \leq 1 \right\} \text{,}
  \end{equation*}
  then $\mathrm{mkD}_{\Lip_\A}$ is a distance on the space $\mathcal{B}_1(\A,\B)$ of unit preserving bounded linear maps, which, on any norm-bounded subset, metrizes the initial topology induced by the family of seminorms:
  \begin{equation*}
    \left\{ \alpha \in \mathcal{B}_1 \mapsto \norm{\alpha(a)}{\B} : a \in \A  \right\} \text{.}
  \end{equation*}
\end{theorem}

\begin{proof}
  Let $(\alpha_n)_{n\in\N}$ be a sequence of unit preserving linear maps converging to some unital linear map $\alpha_\infty$ for $\KantorovichDist{\Lip_\A}{}{}$, and for which there exists some $B > 0$ such that for all $n\in\N\cup\{\infty\}$, we have $\opnorm{\alpha_n}{\A}{\B} \leq B$, where $\opnorm{\cdot}{\A}{\B}$ is the operator norm for linear maps from $\A$ to $\B$. Let $a\in \sa{\A}$ and $\varepsilon > 0$. Since $\dom{\Lip_\A}$ is dense in $\sa{\A}$, there exists $a' \in \dom{\Lip_\A}$ such that $\norm{a-a'}{\A} < \frac{\varepsilon}{3 B}$. By definition, there exists $N\in\N$ such that for all $n\geq N$, we have $\KantorovichDist{\Lip_\A}{\alpha_n}{\alpha_\infty} < \frac{\varepsilon}{3(\Lip_\A(a')+1)}$. Thus:
  \begin{equation*}
    \norm{\alpha_\infty(a)-\alpha_n(a)}{\B} \leq \norm{\alpha_n(a - a')}{\B} + \norm{\alpha_n(a')-\alpha_\infty(a')}{\B} + \norm{\alpha_\infty(a-a')}{\B} < \varepsilon \text{.}
  \end{equation*}
  Thus for all $a\in\sa{\A}$, the sequence $(\alpha_n(a))_{n\in\N}$ converges to $\alpha_\infty(a)$. By linearity, we then conclude $(\alpha_n(a))_{n\in\N}$ converge to $\alpha_\infty(a)$ for $\norm{\cdot}{\B}$.

  Conversely, assume that for all $a\in\A$, the sequence $(\alpha_n(a))_{n\in\N}$ converges  to $\alpha_\infty(a)$ in $\B$, and again assume that there exists $B > 0$ such that for all $n\in\N\cup\{\infty\}$, we have $\opnorm{\alpha_n}{\A}{\B} \leq B$. Let $\varepsilon > 0$ and fix $\mu \in \StateSpace(\A)$. As $\Lip_\A$ is a L-seminorm, $L = \{a\in\sa{\A}:\Lip_\A(a)\leq 1,\mu(a) = 0\}$ is totally bounded. Thus, there exists a finite $\frac{\varepsilon}{3 B}$-dense set $F \subseteq L$ of $L$. As $F$ is finite, by assumption, there exists $N\in\N$ such that for all $n\geq N$ and all $a\in F$, we have $\norm{\alpha_n(a)-\alpha_\infty(a)}{\B} < \frac{\varepsilon}{3}$. If $n\geq N$ and $a\in \sa{\A}$ such that $\Lip_\A(a)\leq 1$ then there exists $a' \in F$ such that $\norm{a-\mu(a)\unit_\A - a'}{\A} < \frac{\varepsilon}{3 B}$, and thus:
\begin{align*}
  \norm{\alpha_n(a)-\alpha_\infty(a)}{\B} 
  &\leq \norm{\alpha_n(a-\mu(a)\unit_\A) - \alpha_\infty(a-\mu(a)\unit_\A)}{\B} \\
  &\leq \norm{\alpha_n(a-\mu(a)\unit_\A - a')}{\B} \\
  &\quad + \norm{\alpha_n(a')-\alpha_\infty(a')}{\B} + \norm{\alpha_\infty(a-\mu(a)\unit_\A)-a'}{\B} \\
  &\leq B\frac{\varepsilon}{3B} + \frac{\varepsilon}{3} + B\frac{\varepsilon}{3B} < \varepsilon \text{.}
\end{align*}
Thus for $n\geq N$, we have $\KantorovichDist{\Lip_\A}{\alpha_n}{\alpha_\infty} < \varepsilon$.
\end{proof}

We now can prove:

\begin{theorem}\label{compactness-thm}
  Let $F$ be a permissible function. Let $(\A_n,\Lip_n,G_n,\delta_n,\alpha_n)_{n\in\N}$ be a sequence of Lipschitz dynamical $F$-systems such that:
  \begin{enumerate}
    \item $(\A_n,\Lip_n)_{n\in\N}$ converges to some {$F$-\qcms s} $(\A_\infty,\Lip_\infty)$ for $\dpropinquity{F}$,
    \item $(G_n,\delta_n)_{n\in\N}$ converges to a proper monoid $(G_\infty,\delta_\infty)$ for $\Upsilon$,
     \item there exists a locally bounded function $D : [0,\infty)\rightarrow [1,\infty)$ such that for all $n\in\N$ and for all $g \in G_n$, we have $\dil{\alpha_n^g}\leq D(\delta_n(e_n,g))$,
    \item for all $\varepsilon > 0$ there exists $\omega > 0$ and $N\in\N$ such that for all $n\geq N$, if $g,h \in G_n$ and $\delta_n(g,h) < \omega$ then:
      \begin{equation*}
        \KantorovichDist{\Lip_n}{\alpha_n^g}{\alpha_n^h} < \varepsilon \text{,}
      \end{equation*}
  \end{enumerate}
  then there exist:
  \begin{itemize}
  \item a strongly continuous action $\alpha_\infty$ of $G_\infty$ on $\A_\infty$ such that $(\A_\infty,\Lip_\infty,G_\infty,\delta_\infty,\allowbreak \alpha_\infty)$ is a Lipschitz dynamical $F$-system, 
  \item for all $n\in\N$, an almost isometry $(\varsigma_n,\varkappa_n) \in \UIso{\varepsilon_n}{(G_n,\delta_n)}{(G,\delta)}{\frac{1}{\varepsilon_n}}$ such that $\lim_{n\rightarrow\infty}\varepsilon_n = 0$,
  \item a strictly increasing sequence $j : \N \rightarrow \N$,
  \item for each $n\in\N$, a tunnel $\tau_n$ from $(\A_\infty,\Lip_\infty)$ to $(\A_n,\Lip_n)$, with \begin{equation*}
      \lim_{n\rightarrow\infty} \tunnelextent{\tau_n} = 0\text{,}
\end{equation*}
  \end{itemize}
    such that for all $a\in\dom{\Lip_\infty}$ and $g \in G_\infty$, with $l\geq \Lip_\infty(a)$ and $K \geq D(\delta_\infty(e,g))$:
    \begin{equation}\label{targetsetforward-cv-eq}
      \lim_{n\rightarrow\infty} \Haus{\norm{\cdot}{\A_\infty}}\left( \targetsetforward{\tau_{j(n)},\alpha_{j(n)}^{\varkappa_{j(n)}(g)}}{a}{l,K} , \{\alpha_\infty^g(a)\} \right) = 0 \text{.}
    \end{equation}

    In particular, for all $g\in G_\infty$, we have $\dil{\alpha_\infty^g}\leq D(\delta_\infty(e,g))$. Furthermore, for all $\varepsilon > 0$ there exist $\omega > 0$ and $N\in\N$ such that if $n\in\N\cup\{\infty\}$ with $n\geq N$, if $g,h \in G_n$ and if $\delta_n(g,h) < \omega$ then $\KantorovichDist{\Lip_n}{\alpha_n^g}{\alpha_n^h}<\varepsilon$.

    If moreover:
\begin{enumerate}
\item for all $n\in\N$, the action $\alpha_n$ is by Lipschitz morphisms, then $\alpha_\infty$ is also an action by Lipschitz morphisms,
\item for all $n\in\N$, the action $\alpha_n$ is by Lipschitz automorphisms, then $\alpha_\infty$ is also an action by Lipschitz automorphisms,
\item for all $n\in\N$, $G_n$ is a group and the action $\alpha_n$ is by full quantum isometries, then $\alpha_\infty$ is also an action by full quantum isometries,
\item for all $n\in\N$, $G_n$ is a compact group and $\alpha_n$ is an ergodic action by full quantum isometries, then $\alpha_\infty$ is also an ergodic action.
\end{enumerate}
\end{theorem}

\begin{proof}
For all $n\in\N$, let $\varepsilon_n = \Upsilon((G_n,\delta_n),(G_\infty,\delta_\infty)) + \frac{1}{n+1}$ and:
\begin{equation*}
  (\varsigma_n,\varkappa_n) \in \UIso{\varepsilon_n}{(G_n,\delta_n)}{(G_\infty,\delta_\infty)}{\frac{1}{\varepsilon_n}}\text{.}
\end{equation*}

Since $G_\infty$ is a proper metric space, it is separable. Let $E$ be a countable dense subset of $G_\infty$ containing the identity element $e$ of $G_\infty$. Let $H$ be the sub-monoid generated by $E$. Since $H$ consists of all the \emph{finite} products of elements of the countable set $E$, it is itself countable. 

Let $\varepsilon > 0$. By assumption, there exists $\omega > 0$ and $N\in\N$ such that if $n\geq N$ and $h,h' \in G_n$ with $\delta_n(h,h') < \omega$ then $\KantorovichDist{\Lip_n}{\alpha_n^h}{\alpha_n^{h'}} < \varepsilon$. Now, fix $g,g' \in H$. There exists $N_1\in \N$ such that if $n\geq N_1$ then $\delta_n(\varkappa_n(g)\varkappa_n(g'),\varkappa_n(gg')) < \omega$ by Assertion (4) of Lemma (\ref{almost-isoiso-lemma}) since $\lim_{n\rightarrow\infty}\varepsilon_n = 0$ and $\lim_{n\rightarrow\infty}\frac{1}{\varepsilon_n} = \infty$.

Let $n\in \N$ with $n\geq\max\{N,N_1\}$. Let $a\in\dom{\Lip_n}$ with $\Lip_n(a)\leq 1$. We now compute:
\begin{align*}
    \left\| \alpha_n^{\varkappa_n(g)}\circ\alpha_n^{\varkappa_n(g')}(a) -  \alpha_n^{\varkappa_n(g g')}(a) \right\|_{\A_n}    &= \left\| \alpha_n^{\varkappa_n(g) \varkappa_n(g')}(a) -  \alpha_n^{\varkappa_n(g g')}(a) \right\|_{\A_n}  \\
    &\leq \KantorovichDist{\Lip_n}{\alpha_n^{\varkappa_n(g)\varkappa_n(g')}}{\alpha_n^{\varkappa_n(gg')}} < \varepsilon \text{.}
\end{align*}

We also record that $\alpha_n^{e_n}$, for $e_n \in G_n$ the identity element of $G_n$, is the identity map.

As a monoid, $H$ is trivially a semigroupoid over the set of a single object, which we take as the identity element of $H$. The domain and codomain maps $c$ and $d$  from $H$ to $\{e\}$ are obviously constant, and the multiplication on $H$ is the composition operation on the semigroupoid $(H,\{e\},c,d,\cdot)$. Thus we fit the hypothesis of \cite[Theorem 2.13]{Latremoliere17c}. Therefore, there exists an action $\alpha_\infty$ of $H$ on $\A$, a strictly increasing sequence $j:\N\rightarrow\N$ and for each $n\in\N$, a tunnel $\tau_n$ from $(\A_\infty,\Lip_\infty)$ to $(\A_{n},\Lip_{n})$ such that $\lim_{n\rightarrow\infty}\tunnelextent{\tau_n} = 0$ and all $a\in\sa{\A_\infty}$ with $\Lip_\infty(a) < \infty$, for all $l \geq \Lip_\infty(a)$, and for all $g \in H$ and $K\geq D(\delta_\infty(e,g))$:
\begin{equation*}
  \lim_{n\rightarrow\infty} \Haus{\norm{\cdot}{\A_\infty}}\left( \targetsetforward{n,g}{a}{l,K} , \{\alpha_\infty^g(a)\} \right) = 0 \text{,}
\end{equation*}
where we use the notation $\targetsetforward{n,g}{\cdot}{\cdot}$ for $\targetsetforward{\tau_{j(n)},\alpha_{j(n)}^{\varkappa_{j(n)}(g)}}{\cdot}{\cdot}$.

Note that by \cite[Theorem 2.13]{Latremoliere17c}, for all $g\in H$, the linear map  $\alpha_\infty^g$ is defined on $\A_\infty$, is unital and positive, and moreover it is a unital *-endomorphism if $\alpha_n$ are actions by unital *-endomorphisms for all $n\in\N$, and even a *-automorphism if $\alpha_n$ is an action by *-automorphisms for all $n\in\N$. Moreover, $\Lip_\infty(\alpha_\infty^g(a))\leq D(\delta_\infty(e,g)) \Lip_\infty(a)$ for all $g\in H$ and $a\in\dom{\Lip_\infty}$, and as a positive unital linear map, $\alpha_\infty^g$ has norm $1$ as a map from $\A_\infty$ to $\A_\infty$ for all $g \in H$.

Let $a\in \dom{\Lip_\infty}$, $l\geq\Lip_\infty(a)$, $l > 0$. Let $\varepsilon > 0$. There exists $\omega > 0$ and $N\in\N$ such that for all $n\geq N$, if $h,h' \in G_n$ and $\delta_n(h,h') < \omega$ then $\KantorovichDist{\Lip_n}{\alpha_n^{h}}{\alpha_n^{h'}} < \frac{\varepsilon}{4 l}$. Let $g,h \in H$ such that $\delta(g,h) < \frac{\omega}{2}$.

Let $N_1 \in \N$ such that for all $n\geq N_1$, we have $\delta_{j(n)}(\varkappa_{j(n)}(g),\varkappa_{j(n)}(h)) < \delta(g,h) + \frac{\omega}{2} = \omega$ by Lemma (\ref{almost-isoiso-lemma}).

Since $D$ is locally bounded, and since:
\begin{equation*}
  \lim_{n\rightarrow\infty}\delta_{j(n)}(e_{j(n)},\varkappa_{j(n)}(g)) = \lim_{n\rightarrow\infty}\delta_{j(n)}(e_{j(n)},\varkappa_{j(n)}(h)) = 0\text{, }
\end{equation*}
there exists $N_2 \in \N$ and $K > 0$ such that for all $n\geq N_2$ we have $D(\delta_{j(n)}(e_{j(n)},\varkappa_{j(n)}(g))) \leq K$ and $D(\delta_{j(n)}(e_{j(n)},\varkappa_{j(n)}(g))) \leq K$.

Let $N_3\in\N$ such that for all $n\geq N_3$, we have:
\begin{align*}
  \Haus{\norm{\cdot}{\A_\infty}}\left(\{\alpha_\infty^g(a)\}, \targetsetforward{n,g}{a}{l,K}\right) &< \frac{\varepsilon}{4}\\ 
\intertext{and }
\Haus{\norm{\cdot}{\A_\infty}}\left(\{\alpha_\infty^h(a)\}, \targetsetforward{n,h}{a}{l,K}\right) &< \frac{\varepsilon}{4}\text{.}
\end{align*}

Let $N_4 \in \N$ such that for all $n\geq N_4$, we have $\tunnelextent{\tau_{j(n)}} < \frac{\varepsilon}{8 K l}$.

Let $n\geq \max\{N,N_1,N_2,N_3,N_4\}$. Now let $a_n \in \targetsettunnel{\tau_{j(n)}}{a}{l}$. Let $b_n = \alpha_{j(n)}^{\varkappa_{j(n)}(g)}(a_n)$ and $c_n = \alpha_{j(n)}^{\varkappa_{j(n)}(h)}(a_n)$. Let $d_n \in \targetsettunnel{\tau_{j(n)}^{-1}}{b_n}{K l}$ and $f_n \in \targetsettunnel{\tau_{j(n)}^{-1}}{c_n}{K l}$. We then have:
\begin{align*}
  \norm{\alpha_\infty^g(a) - \alpha_\infty^h(a)}{\A_\infty} &\leq \norm{\alpha_\infty^g(a) - d_n}{\A_\infty} + \norm{d_n - f_n}{\A_\infty} + \norm{f_n - \alpha_\infty^h(a)}{\A_\infty}\\
                                       &\leq \frac{\varepsilon}{4} + \norm{b_n - c_n}{\A_{j(n)}} + 2 K l \tunnelextent{\tau_{j(n)}} + \frac{\varepsilon}{4} \\
                                       &\leq \frac{\varepsilon}{2} + \norm{\alpha_{j(n)}^{\varkappa_{j(n)}(g)}(a_n) - \alpha_{j(n)}^{\varkappa_{j(n)}(h)}(a_n)}{\A_{j(n)}} + 2 K l \tunnelextent{\tau_{j(n)}} \\
                                       &\leq  \frac{3\varepsilon}{4} + l \KantorovichDist{\Lip_{j(n)}}{\alpha_{j(n)}^{\varkappa_{j(n)}(g)} }{\alpha_{j(n)}^{\varkappa_{j(n)}(h)}}\\
                                       &< \varepsilon \text{.}
\end{align*}
Hence, $g \in H \mapsto \alpha_\infty^g(a)$ is uniformly continuous over a dense subset $H$ of the complete space $(G_\infty,\delta_\infty)$. Hence, it admits a unique uniformly continuous extension to $G_\infty$, but we are going to prove a little more. By assumption, for all $g\in H$, we have $\opnorm{\alpha_\infty^g}{}{\A_\infty} \leq 1$ where $\opnorm{\cdot}{}{\A_\infty}$ is the operator norm for a linear map on $\A_\infty$. Let $a \in \sa{\A_\infty}$ and $\varepsilon > 0$. There exists $a'\in\dom{\Lip_\infty}$ with $\norm{a-a'}{\A_\infty} < \frac{\varepsilon}{3}$ (using the density of $\dom{\Lip_\infty}$). Since $g \in H \mapsto \alpha_\infty^g(a')$ is uniformly continuous, there exists $\omega > 0$ such that if $g,h \in H$ and $\delta_\infty(g,h) < \omega$ then $\norm{\alpha_\infty^g(a') - \alpha_\infty^h(a')}{\A_\infty} < \frac{\varepsilon}{3}$.  Therefore: 
\begin{multline*}
  \norm{\alpha_\infty^g(a) - \alpha_\infty^h(a)}{\A_\infty} \\
  \leq \norm{\alpha_\infty^g(a) - \alpha_\infty^g(a')}{\A_\infty} + \norm{\alpha_\infty^g(a') - \alpha_\infty^h(a')}{\A_\infty} \\
  + \norm{\alpha_\infty^h(a') - \alpha_\infty^h(a)}{\A_\infty} \leq \varepsilon \text{.}
\end{multline*}
Hence for all $a\in\sa{\A_\infty}$, the map $g \in H \mapsto \alpha_\infty^g(a)$ is uniformly continuous on $H$, which is dense in the complete metric space $(G_\infty,\delta_\infty)$. Hence, it admits a unique equicontinuous extension to $\sa{\A_\infty}$, which we still denote by $\alpha_\infty^g(a)$. 

Moreover, $\alpha_\infty^g$ is a unital positive $\R$-linear map (hence, of norm $1$) and if for all $n\in\N$, the action $\alpha_n$ is by Lipschitz endomorphism, then $\alpha_\infty^g$ is a Jordan-Lie morphism for all $g \in G_\infty$.

Another consequence of this observation is that $g \in G_\infty \mapsto \dil{\alpha_\infty^g}$ is locally bounded. Indeed, let $g \in G_\infty$. There exists $\omega>0$ and $M > 0$ such that if $\delta_\infty(g,h)<\omega$ then $D(\delta_\infty(e,h)) \leq M$. Assume now $\delta_\infty(g,h) < \frac{\omega}{2}$. There exists a sequence $(h_n)_{n\in\N}$ in $H$ converging to $h$ --- we may as well assume that $\delta_\infty(h,h_n) \leq \frac{\omega}{2}$, and thus $\delta_\infty(g,h_n) < \omega$ for all $n\in\N$. We conclude, since $\Lip_\infty$ is lower semi-continuous, that:
\begin{equation*}
\Lip_\infty(\alpha_\infty^h(a)) \leq \liminf_{n\rightarrow\infty} \Lip_\infty(\alpha_\infty^{h_n}(a)) \leq M \Lip_\infty(a)
\end{equation*}
for all $a\in \dom{\Lip_\infty}$. Hence, $g \in G_\infty\mapsto \dil{\alpha_\infty^g}$ is locally bounded as well (though not a priori using the function $D$).

Setting, $\alpha_\infty^g(a) = \alpha_\infty^g\left(\frac{a+a^\ast}{2}\right) + i \alpha_\infty^g\left(\frac{a-a^\ast}{2i}\right)$ for $a\in A_\infty$ and $g \in G_\infty$, we check that $\alpha_\infty^g$ is a positive unital linear map (hence still of norm $1$) on $\A_\infty$, which is a unital *-endomorphism of $\A_\infty$ if $\alpha_\infty^g$ is a Jordan-Lie morphism, as seen for instance in Claim (5.18) of the proof of \cite[Theorem 5.13]{Latremoliere13}.

It is also easy to check that $g\in G_\infty \mapsto \alpha_\infty^g$ is a monoid action on $\A_\infty$ since $g \in H \mapsto \alpha_\infty^g$ is, and since the multiplication on $G_\infty$ is assumed continuous. By construction, the action $\alpha_\infty$ is strongly continuous on $\sa{\A_\infty}$ and hence on $\A_\infty$ by an immediate computation. If $G_\infty$ is a group, since $\alpha_\infty^e$ is the identity, we then conclude that the action $\alpha_\infty$ of $G_\infty$ is an action by invertible Lipschitz linear maps by \cite[Theorem 2.13]{Latremoliere17c}.

Now, we show that Expression (\ref{targetsetforward-cv-eq}) holds for all $g \in G_\infty$ rather than all $g\in H$. Let $a\in \dom{\Lip_\infty}$ and $l>\Lip_\infty(a)$. Let $g \in G_\infty$. Since both $D$ and $g\in G_\infty \mapsto \dil{\alpha_\infty^g}$  are locally bounded, there exists $\omega_1 > 0$ and $M > 0$ such that if $\delta_\infty(g,h) < \omega_1$ then $\max\{ \dil{\alpha_\infty^h}, D(\delta_\infty(e,h)) \} \leq M$.

Let $\varepsilon > 0$. By uniform continuity of $h\in G_\infty\mapsto\alpha_\infty^h(a)$, there exists $\omega > 0$ such that if $h,h' \in G$ and $\delta_n(h,h') < \omega$ then $\norm{\alpha_\infty^{h}(a) - \alpha_\infty^{h'}(a)}{\A_\infty} < \frac{\varepsilon}{4}$.  By density of $H$, there exists $h \in H$ such that $\delta_\infty(g,h) < \min\{\omega,\omega_1\}$.

Now, there exists $N_1\in\N$ such that for all $n\geq N_1$, we have:
\begin{equation*}
  \Haus{\norm{\cdot}{\A_\infty}}\left(\{\alpha_\infty^h(a)\}, \targetsetforward{n,h}{a}{l,M} \right) < \frac{\varepsilon}{4} \text{.}
\end{equation*}

Moreover, there exists $N_2 \in \N$ such that if $n\geq N_2$ then $\tunnelextent{\tau_{n}}\leq \frac{\varepsilon}{8 l M}$, and $N_3 \in \N$ such that if $n\geq N_3$ then $\delta_{j(n)}(\varkappa_{j(n)}(g),\varkappa_{j(n)}(h)) < \min\{\omega,\omega_1\}$.

Let $n\geq\max\{N_1,N_2,N_3\}$.

Let $b_n\in\targetsetforward{n,g}{a}{M,l}$. There exists $a_n \in \targetsettunnel{\tau_{n}}{a}{l}$ such that:
\begin{equation*}
  b_n \in \targetsettunnel{\tau_{n}^{-1}}{\alpha_{j(n)}^{\varkappa_{j(n)}(g)}(a_n)}{M l} \text{.}
\end{equation*}
Let $c_n \in \targetsettunnel{\tau_{n}^{-1}}{\alpha_{j(n)}^{\varkappa_{j(n)}(h)}(a_n)}{M l}$. Note that $c_n \in \targetsetforward{n,h}{a}{l,M}$. We now estimate:
\begin{align*}
  \norm{\alpha_\infty^g(a) - b_n}{\A_\infty} &\leq \norm{\alpha_\infty^g(a) - \alpha_\infty^h(a)}{\A_\infty} + \norm{\alpha_\infty^h(a) - c_n}{\A_\infty} + \norm{c_n - b_n}{\A_\infty} \\
  &\leq \frac{\varepsilon}{4} + \frac{\varepsilon}{4} + \norm{\alpha_{j(n)}^{\varkappa_{j(n)}(h)}(a_n) - \alpha_{j(n)}^{\varkappa_{j(n)}(g)}(a_n)}{\A_{j(n)}} + 2 M l \tunnelextent{\tau_{j(n)}}\\
  &< \varepsilon \text{.}
\end{align*}
Hence:
\begin{equation*}
  \lim_{n\rightarrow\infty} \Haus{\norm{\cdot}{\A_\infty}}\left(\{\alpha_\infty^g(a)\}, \targetsetforward{n,g}{a}{l,M}\right) = 0 \text{,}
\end{equation*}
as desired.

Now, Expression (\ref{targetsetforward-cv-eq}) holds for any $D\geq D(\delta_\infty(e,g))$ --- since all forward target sets $\targetsetforward{n,g}{a}{l,L}$ have diameter converging to $0$ and are not empty for $L \geq D(\delta_\infty(e,g))$, and $\targetsetforward{n,g}{a}{l,L}\subseteq \targetsetforward{n,g}{a}{l,L'}$ with $L'\geq L\geq D(\delta_\infty(e,g))$ --- so we may apply, for instance, \cite[Claim 6.13]{Latremoliere16c}.

In turn, this proves that $\dil{\alpha_\infty^g}\leq D(\delta_\infty(e,g))$ by lower semi-continuity of $\Lip_\infty$.

We make one last observation. Let $\varepsilon > 0$. There exist $\omega > 0$ and $N\in\N$ such that if $n\geq N$, $g,h \in G_n$ and $\delta_n(g,h) < \omega$ then $\KantorovichDist{\Lip_n}{\alpha_n^g}{\alpha_n^h} < \varepsilon$. For all $g,h \in H$ with $\delta_\infty(g,h) < \omega$, then there exists $N_1\in\N$ with $N_1\geq N$ such that if $n\geq N_1$ then $\delta_{j(n)}(\varkappa_{j(n)}(g),\varkappa_{j(n)}(h)) < \omega$ since $\left|\delta_{j(n)}(\varkappa_{j(n)}(g),\varkappa_{j(n)}(h)) - \delta(g,h)\right|\xrightarrow{n\rightarrow\infty} 0$.  By \cite[Theorem 2.13]{Latremoliere17c}, we conclude $\KantorovichDist{\Lip_\infty}{\alpha_\infty^g}{\alpha_\infty^h} < \varepsilon$.

Let now $g,h \in G_\infty$ with $\delta_\infty(g,h) < \omega$. By density of $H$, there exists two sequences $(g_n)_{n\in\N}$, $(h_n)_{n\in\N} \in H$ such that $\lim_{n\rightarrow\infty}g_n = g$ and $\lim_{n\rightarrow\infty}h_n = h$. Let $a\in\dom{\Lip_\infty}$ with $\Lip_\infty(a)\leq 1$. We have:
\begin{align*}
  \norm{\alpha_\infty^g(a)-\alpha_\infty^h(a)}{\A_\infty} 
  = \lim_{n\rightarrow\infty} \norm{\alpha_\infty^{g_n}(a) - \alpha_\infty^{h_n}(a)}{\A_\infty}
  \leq \varepsilon \text{.}
\end{align*}

This concludes our proof, as the remaining arguments regarding full quantum isometries and ergodicity follows as in  \cite[Theorem 3.14]{Latremoliere17c}.
\end{proof}

We now can state a sufficient condition for a certain kind of compactness.

\begin{theorem}\label{completeness-thm}
  Let $(\A,\Lip)$ be an {$F$-{\qcms}} and $(G,\delta)$ be a proper monoid. Let $(\A_n,\Lip_n,G_n,\delta_n,\alpha_n)_{n\in\N}$ be a sequence of Lipschitz dynamical systems and let $D : [0,\infty) \rightarrow [1,\infty)$ be a locally bounded function such that:
  \begin{enumerate}
    \item for all $n\in\N$ and $g\in G_n$, we have $\dil{\alpha_n^g} \leq D(\delta_n(e_n,g))$,
    \item $\lim_{n\rightarrow\infty} \Upsilon((G_n,\delta_n),(G,\delta)) = 0$ ,
    \item $\lim_{n\rightarrow\infty} \dpropinquity{}((\A_n,\Lip_n),(\A,\Lip)) = 0$,
    \item for all $\varepsilon>0$, there exists $\omega > 0$ and $N\in\N$ such that if $n\geq N$ and if $g,h \in G_n$ with $\delta_n(g,h) < \omega$, then $\KantorovichDist{\Lip_n}{\alpha_n^g}{\alpha_n^h}{} < \varepsilon$,
  \end{enumerate}
  then there exists a strictly increasing function $j : \N \rightarrow \N$ and a Lipschitz dynamical system $(\A,\Lip,G,\delta,\alpha)$ such that:
  \begin{equation*}
    \covpropinquity{}((\A_{j(n)},\Lip_{j(n)},G_{j(n)},\delta_{j(n)},\alpha_{j(n)}),(\A,\Lip,G,\delta,\alpha)) \xrightarrow{n\rightarrow\infty} 0 \text{.}
  \end{equation*}
As the action $\alpha$ is given by Theorem (\ref{compactness-thm}), it enjoys the properties described in the conclusion of that theorem.
\end{theorem}

\begin{proof}
  By Theorem (\ref{compactness-thm}), there exists a strongly continuous action $\alpha$ of $(G,\delta)$ on $\A$, a strictly increasing function $j : \N \rightarrow \N$ and a sequence:
\begin{equation*}
  \left(\tau_n \right)_{n\in\N} = \left(\D_n,\Lip^n,\pi_n,\rho_n\right)_{n\in\N}
\end{equation*}
of tunnels where $\tau_n$ is a tunnel from $(\A,\Lip)$ to $(\A_{n},\Lip_{n})$ for all $n\in\N$, such that $(\A,\Lip,G,\delta,\alpha)$ is a Lipschitz dynamical $F$-system and for all $a\in\dom{\Lip}$, $l\geq\Lip(a)$, and $g \in G$, $M \geq D(\delta(e,g))$:
  \begin{equation*}
    \lim_{n\rightarrow\infty} \Haus{\norm{\cdot}{\A}}\left( \left\{ \alpha^g(a) \right\}, \targetsetforward{\tau_{j(n)},\alpha_{j(n)}^{\varkappa_{j(n)}(g)}}{a}{l,M} \right) = 0 \text{.}
  \end{equation*}

By replacing the original sequences of tunnels, Lipschitz dynamical systems, and almost isometries by their subsequence indexed by $j$, we will dispense with writing $j$ in the rest of this proof. We also write $\targetsetforward{n,g}{\cdot}{\cdot} = \targetsetforward{\tau_n,\alpha_n^{\varkappa_n(g)}}{\cdot}{\cdot}$ for all $n\in \N$ and $g\in G$.

Let $\varepsilon \in (0,1)$. By assumption, there exists $\omega > 0$ and $N\in\N$ such that for all $n\geq N$ and $h,h' \in G_n$, we have $\delta_n(h,h') < \omega$ then $\KantorovichDist{\Lip_n}{\alpha_n^h}{\alpha_n^{h'}} < \frac{\varepsilon}{16}$.

  Fix $\mu \in \StateSpace(\A)$. By compactness, there exists a $\frac{\varepsilon}{32}$-dense finite subset $\mathcal{L}$ (for $\norm{\cdot}{\A}$) of $\{a\in\dom{\Lip}:\Lip(a)\leq 1,\mu(a)=0\}$. As $G$ is proper, the closed ball $G\left[\frac{2}{\varepsilon}\right]$ is compact, so there exists a finite $\frac{\omega}{2}$-dense subset $F$ of $G\left[\frac{2}{\varepsilon}\right]$ for $\delta$. We assume without loss of generality that the unit $e$ of $G$ is in $F$. Let:
\begin{equation*}
  \mathcal{F} = \left\{ \alpha^g(a) \middle\vert a \in \mathcal{L}, g \in F \right\} \text{.}
\end{equation*}
Note that $\mathcal{F}$ is finite by construction.

Since $\left[ 0, \frac{2}{\varepsilon} \right]$ is compact, and since $D$ is locally bounded, we conclude that there exists $M \geq 1$ such that $D(r) \leq M$ for all $r \in \left[ 0, \frac{2}{\varepsilon} \right]$.

There exists $N_1\in\N$ such that for all $n\geq N_1$,
\begin{equation*}
  \Haus{\norm{\cdot}{\A}}\left(\{\alpha^g(a)\},\targetsetforward{n,g}{a}{l, M} \right) < \frac{\varepsilon}{16}
\end{equation*}
for all $a\in\mathcal{L}$ and $g \in F$.

 There exists $N_2 \in\N$ such that for all $n\geq N_2$, we have $\tunnelextent{\tau_{n}} < \frac{\varepsilon}{64 M}$.

 Let $\eta = \min\left\{ \frac{\omega}{2},\frac{\varepsilon}{2} \right\} > 0$. There exists $N_3 \in \N$ such that for all $n\geq N_3$, we have that:
 \begin{equation*}
   (\varsigma_n,\varkappa_n) \in \UIso{\eta}{(G_n,\delta_n)}{(G,\delta)}{\frac{1}{\eta}}\text{.}
\end{equation*}

 Fix $n\geq \max\{ N, N_1, N_2, N_3 \}$.

Let $\varphi \in \StateSpace(\A)$. There exists $\psi \in \StateSpace(\A_n)$ such that $\Kantorovich{\Lip^n}(\varphi\circ\pi_n,\psi\circ\rho_n) < \frac{\varepsilon}{8 M}$. (We note that all the computations below are also valid if we start with $\psi \in \StateSpace(\A_n)$ and obtain $\varphi\in\StateSpace(\A)$ such that $\Kantorovich{\Lip^n}(\varphi\circ\pi_n,\psi\circ\rho_n) < \frac{\varepsilon}{8 M}$).

Let $d\in\dom{\Lip^n}$ with $\Lip^n(d) \leq 1$ and $\mu\circ\pi_n(d) = 0$. Let $g \in G\left[\frac{2}{\varepsilon}\right]$. Write $a = \pi_n(d)$ and $b  = \rho_n(d)$ --- note that $\mu(a)  = 0$. First, by definition of $\mathcal{L}$ and $F$, there exists $a' \in \mathcal{L}$ such that:
\begin{equation*}
  \norm{a-a'}{\A} <\frac{\varepsilon}{32}
\end{equation*}
and $h \in F$ such that $\delta(g,h) < \frac{\omega}{2}$, and therefore $\delta_n(\varkappa_n(g),\varkappa_n(h)) < \allowbreak \delta_n(g,h) + \frac{\omega}{2} < \omega$, which leads to:
\begin{equation*}
  \norm{\alpha^g(a) - \alpha^{h}(a)}{\A} < \frac{\varepsilon}{16}\text{ and }\norm{\alpha_n^{\varkappa_n(g)}(b) - \alpha_n^{\varkappa_n(h)}(b)}{\A_n} < \frac{\varepsilon}{16}
\end{equation*}

Let now $b' \in \targetsettunnel{\tau_n}{a'}{1}$. We compute:
\begin{equation*}
  \norm{b-b'}{\A_n} \leq \norm{a-a'}{\A} + 2 \tunnelextent{\tau_n} < \frac{\varepsilon}{16} \text{.}
\end{equation*}

Therefore:
\begin{align*}
  \left|\varphi(\alpha^g(a))-\varphi(\alpha^h(a'))\right|&\leq\left|\varphi(\alpha^g(a)-\alpha^g(a'))\right| + \left|\varphi(\alpha^g(a') - \alpha^h(a'))\right|\\
&\leq \norm{a-a'}{\A} + \KantorovichDist{\Lip}{\alpha^g}{\alpha^h} \\
&<\frac{\varepsilon}{8} \text{.}
\end{align*}
Similarly, we also have:
\begin{equation*}
  \left|\psi(\alpha_n^{\varkappa_n(g)}(b)) - \psi(\alpha_n^{\varkappa_n(h)}(b'))\right| < \frac{\varepsilon}{8}\text{,}
\end{equation*}
and therefore:
\begin{align*}
  \left|\varphi(\alpha^g(a)) - \psi(\alpha_n^{\varkappa_n(g)}(b))\right|  < \frac{\varepsilon}{4} + \left|\varphi(\alpha^h(a')) - \psi(\alpha_n^{\varkappa_n(h)}(b'))\right| \text{.}
\end{align*}

Note that $b'\in\targetsettunnel{\tau_n}{a}{1}$ by definition. Therefore, $c = \alpha_n^{\varkappa_n(h)}(b') \in \targetsetimage{n,h}{a'}{1,M}$. Let $x\in\targetsettunnel{\tau^{-1}_n}{c}{1,M}$. By construction, $x\in \targetsetforward{n,h}{a'}{1,M}$ and thus $\norm{\alpha^h(a') - x}{\A} < \frac{\varepsilon}{8}$. Moreover, by definition of target set, there exists $d\in\D$ such that $\pi_n(d) = x$ and $\rho_n(d) = c = \alpha_n^{\varkappa_n(h)}(b')$ and $\Lip^n(d) \leq M$.
\begin{align*}
  \left| \varphi(\alpha^h(a')) - \psi(\alpha_n^{\varkappa_n(h)}(b')) \right| 
  &\leq \left|\varphi(\alpha^h(a') - x)\right| + \left|\varphi(x) - \psi(\alpha_n^{\varkappa_n(h)}(b'))\right| \\
  &\leq  \norm{\alpha^h(a') - x}{\A}  + \left|\varphi\circ\pi_n(d) - \psi\circ\rho_n(d)\right|\\
  &\leq \frac{\varepsilon}{4} \text{.}
\end{align*}

Thus, for all $g \in G\left[\frac{2}{\varepsilon}\right]$, for all $\varphi\in\StateSpace(\A)$, and for all $d \in \dom{\Lip^n}$ with $\Lip^n(d) \leq 1$ and $\mu\circ\pi_n(d) =0$, we have proven:
\begin{equation*}
  \left|\varphi\circ\alpha^g\circ\pi_n(d) - \psi\circ\alpha_n^{\varkappa_n(g)}\circ\rho_n(d)\right| < \frac{\varepsilon}{2} \text{.}
\end{equation*}

Now, let $g \in G_n\left[\frac{1}{\varepsilon}\right]$ and $\varphi\in\StateSpace(\A)$. Again, there exists $\psi \in \StateSpace(\A_n)$ such that $\Kantorovich{\Lip^n}(\varphi\circ\pi_n,\psi\circ\rho_n) < \frac{\varepsilon}{8 M}$. By our previous work, since $\varkappa_n(\varsigma_n(g)) \in G\left[\frac{1}{\varepsilon}+\varepsilon\right]\subseteq G\left[\frac{2}{\varepsilon}\right]$ by Lemma (\ref{almost-isoiso-lemma}), we note that:
\begin{equation*}
\left|\varphi(\alpha^{\varsigma_n(g)}(d)) - \psi(\alpha_n^{\varkappa_n(\varsigma_n(g))}(d))\right| < \frac{\varepsilon}{2} \text{.}
\end{equation*}

Therefore, we conclude:
\begin{align*}
  \left|\varphi(\alpha^{\varsigma_n(g)}(d)) - \psi(\alpha_n^g(d))\right| &\leq \left|\varphi(\alpha^{\varsigma_n(g)}(d)) - \psi(\alpha_n^{\varkappa_n(\varsigma_n(g))}(d))\right| \\
&\quad + \norm{\alpha_n^{g}(a) - \alpha_n^{\varsigma_n\circ\varkappa_n(g)}(a)}{\A_n}\\
&\leq \varepsilon\text{.}
\end{align*}

Now, let $d\in \dom{\Lip^n}$ with $\Lip^n(d) \leq 1$. Then we note that $\mu\circ\pi_n(d - \mu\circ\pi_n(d)\unit_{\D_n}) = 0$ and $\Lip^n(d-\mu\circ\pi_n(d)\unit_{\D_n}) \leq 1$ as L-seminorms vanish on scalars, and thus we have just proven:
\begin{equation*}
  \left|\varphi(\alpha^{\varsigma_n(g)}(d)) - \psi(\alpha_n^g(d))\right| 
    = \left|\varphi(\alpha^{\varsigma_n(g)}(d - \mu\circ\pi_n(d)\unit_{\D_n})) - \psi(\alpha_n^g(d - \mu\circ\pi_n(d)\unit_{\D_n}))\right|
    \leq \varepsilon \text{.}
\end{equation*}

Our computation is again valid if we start with $\psi\in\StateSpace(\A_n)$ and choose $\varphi\in\StateSpace(\A)$ with $\Kantorovich{\Lip^n}(\varphi\circ\pi_n,\psi\circ\rho_n)<\frac{\varepsilon}{8 M}$.

  Therefore $\tunnelreach{\tau_{n}}{\varepsilon} < \varepsilon$. Since $\tunnelextent{\tau_n} < \frac{\varepsilon}{64 M} < \varepsilon$, we conclude that $\tunnelmagnitude{\tau_n}{\varepsilon} < \varepsilon$. Hence we have completed our proof.
\end{proof}

We thus can conclude on a sufficient condition for convergence of Cauchy sequences for the covariant propinquity:

\begin{corollary}
  Let $F$ be permissible and continuous and let $D : [0,\infty)\rightarrow [1,\infty)$ be a locally bounded function. Let $(\A_n,\Lip_n,G_n,\delta_n,\alpha_n)_{n\in\N}$ be a sequence of Lipschitz dynamical $F$-systems and $(\varepsilon_n)_{n\in\N}$ a sequence of positive real numbers such that for all $n\in\N$, there exists $\varepsilon_n > 0$ and $(\varsigma_n,\varkappa_n) \in \UIso{\varepsilon_n}{G_n}{G_{n+1}}{\frac{1}{\varepsilon_n}}$ and:
  \begin{enumerate}
  \item $\sum_{n=0}^\infty \varepsilon_n < \infty$,
  \item for all $n\in\N$ and $g \in G_n$:
    \begin{equation*}
      \left(\begin{cases}
          g_n = e_n \text{ if $n < N$,}\\
          g_n = g    \text{ if $n = N$,}\\
          g_n = \varsigma_n(g_{n-1}) \text{ if $n>N$.}
        \end{cases}\right)_{n\in\N} \in \mathcal{R}((G_n,\delta_n)_{n\in\N})\text{,}
    \end{equation*}
  \item $\forall n\in\N \quad \dpropinquity{}((\A_n,\Lip_n),(\A_{n+1},\Lip_{n+1})) < \varepsilon_n$,
    \item $\forall n \in \N \quad g \in G_n \quad \Lip_n\circ\alpha_n^{g} \leq D(\delta_n(e_n,g))\Lip_n$,
    \item for all $\varepsilon>0$, there exists $\omega > 0$ and $N\in\N$ such that if $n\geq N$ and if $g,h \in G_n$ with $\delta_n(g,h) < \omega$, then $\KantorovichDist{\Lip_n}{\alpha_n^g}{\alpha_n^h}{} < \varepsilon$,
  \end{enumerate}
then there exists a Lipschitz dynamical $F$-system $(\A,\Lip_\A,G,\delta,\alpha)$ such that:
\begin{equation*}
  \lim_{n\rightarrow\infty} \covpropinquity{}((\A_n,\Lip_n,G_n,\delta_n,\alpha_n),(\A,\Lip_\A,G,\delta,\alpha)) = 0 \text{.}
\end{equation*}
  Moreover, if for all $n\in\N$, the action $\alpha_n$ is by *-endomorphisms, then so is the action $\alpha$.
\end{corollary}

\begin{proof}
  The sequence $(\A_n,\Lip_n)_{n\in\N}$ is a Cauchy sequence for the dual $F$-propinquity, which is complete by \cite{Latremoliere13b,Latremoliere15} since $F$ is continuous, so there exists a {\qcms} $(\A,\Lip_\A)$ such that $\lim_{n\rightarrow\infty} \dpropinquity{}((\A_n,\Lip_n),(\A,\Lip_\A)) = 0$.
Moreover by Theorem (\ref{Upsilon-completeness-thm}), there exists a proper monoid $(G,\delta)$ such that:
\begin{equation*}
  \lim_{n\rightarrow\infty} \Upsilon((G_n,\delta_n),(G,\delta)) = 0\text{.}
\end{equation*}
By Theorem (\ref{compactness-thm}), which we now may apply, there exists a strictly increasing function $j : \N \rightarrow \N$ such that:
\begin{equation*}
  (\A_{j(n)},\Lip_{j(n)},G_{j(n)},\delta_{j(n)},\alpha_{j(n)})_{n\in\N}
\end{equation*}
converges to a Lipschitz dynamical $F$-system $(\A,\Lip_\A,G,\delta,\alpha)$.

Since a Cauchy sequence with a convergent subsequence converges, our corollary is now proven.
\end{proof}

A particular consequence of our work is a simpler result concerning Lipschitz dynamical systems with bi-invariant metrics, since regularity is no longer an hypothesis.

\begin{corollary}
  Let $(\A_n,\Lip_n,G_n,\delta_n,\alpha_n)_{n\in\N}$ be a sequence of Lipschitz dynamical systems with $\delta_n$ bi-invariant for all $n\in\N$. If:
  \begin{enumerate}
    \item $(\A_n,\Lip_n,G_n,\delta_n,\alpha_n)_{n\in\N}$ is Cauchy for $\covpropinquity{}$,
    \item there exists a locally bounded function $D: [0,\infty)\rightarrow[1,\infty)$ such that:
      \begin{equation*}
        \forall n\in\N, g \in G_n \quad \Lip_n\circ\alpha_n^g \leq D(\delta_n(e_n,g)) \Lip_n
      \end{equation*}
      where $e_n \in G_n$ is the identity element of $G_n$ for all $n\in\N$,
    \item for all $\varepsilon>0$, there exists $\omega > 0$ and $N\in\N$ such that if $n\geq N$ and if $g,h \in G_n$ with $\delta_n(g,h) < \omega$, then $\KantorovichDist{\Lip_n}{\alpha_n^g}{\alpha_n^h}{} < \varepsilon$,
  \end{enumerate}
then there exists a Lipschitz dynamical system $(\A,\Lip_\A,G,\delta,\alpha)$ such that:
\begin{equation*}
  \lim_{n\rightarrow\infty} \covpropinquity{}((\A_n,\Lip_n,G_n,\delta_n,\alpha_n),(\A,\Lip_\A,G,\delta,\alpha)) = 0 \text{.}
\end{equation*}
  Moreover, if for all $n\in\N$, the action $\alpha_n$ is by *-endomorphisms, then so is the action $\alpha$.
\end{corollary}

We also record the implications of our work on Lipschitz C*-dynamical systems, under the strong assumption of Corollary (\ref{Upsilon-completeness-cor}).

\begin{corollary}
    Let $F$ be permissible and continuous. If $(\A_n,\Lip_n,G_n,\delta_n,\alpha_n)_{n\in\N}$ is a Cauchy sequence of Lipschitz C*-dynamical $F$-systems such that:
\begin{itemize}
  \item for all $\varepsilon > 0$ there exists $\omega > 0$ and $N\in \N$ such that for all $n\geq N$, if $g,h \in G_n$ with $\delta_n(g,h) < \omega$ then $\delta_n\left(g^{-1},h^{-1}\right) < \varepsilon$,
  \item for all $\varepsilon>0$, there exists $\omega > 0$ and $N\in\N$ such that if $n\geq N$ and if $g,h \in G_n$ with $\delta_n(g,h) < \omega$, then $\KantorovichDist{\Lip_n}{\alpha_n^g}{\alpha_n^h}{} < \varepsilon$,
  \item there exists a locally bounded function $D : [0,\infty)\rightarrow [1,\infty)$ such that for all $n\in\N$, $g\in G_n$, and with $e_n \in G_n$ the identity of $G_n$, we have $\Lip_n\circ\alpha_n^g\leq D(\delta_n(e_n,g))\Lip_n$,
\end{itemize}
then there exists a Lipschitz C*-dynamical $F$-system $(\A,\Lip,G,\delta,\alpha)$ such that:
\begin{equation*}
  \lim_{n\rightarrow\infty} \covpropinquity{F}((\A_n,\Lip_n,G_n,\delta_n,\alpha_n),(\A,\Lip,G,\delta,\alpha)) = 0 \text{.}
\end{equation*}
\end{corollary}

\begin{proof}
  This follows from Theorem (\ref{completeness-thm}) and Corollary (\ref{Upsilon-completeness-cor}).
\end{proof}

We conclude this section with two observations. First, there are many natural complete classes of Lipschitz dynamical systems for the covariant propinquity. Let $F$ be a continuous admissible function, $D:[0,\infty)\rightarrow [1,\infty)$ be a locally bounded function, and $K : [0,\infty) \rightarrow [0,\infty]$ be a function with $K(0) = 0$ and $K$ continuous at $0$. Let us denote by $\mathcal{C}$ the class of all Lipschitz dynamical $F$-systems $(\A,\Lip,G,\delta,\alpha)$ such that:
\begin{itemize}
\item $(G,\delta)$ is a proper monoid with $\delta$ bi-invariant,
\item for all $g,h \in G$, we have $\KantorovichDist{\Lip}{\alpha^g}{\alpha^h} \leq K(\delta(g,h))$.
\item for all $g \in G$, $\Lip\circ\alpha^g \leq D(\delta(e,g))\Lip$ (where $e\in G$ is the identity element of $G$).
\end{itemize}
Let $\mathcal{C}^\ast$ be the subclass of $\mathcal{C}$ consisting of Lipschitz C*-dynamical systems. Then Theorem (\ref{completeness-thm}) and its corollaries prove that both $\mathcal{C}$ and $\mathcal{C}^\ast$ are complete for $\covpropinquity{F}$. These are but certain possible complete classes: for instance, we could relax the hypothesis of working with bi-invariant metrics by asking that right translations in our monoids are Lipschitz with some uniform bound on the Lipschitz constant.

Second, there is a natural way to metrize compact monoids and groups acting on a {\qcms}. We discuss this point in the case of compact groups. Let us be given a compact group $G$ and a strongly continuous action $\alpha$ of $G$ by Lipschitz automorphisms of a {\qcms} $(\A,\Lip)$. Then we can define:
\begin{equation*}
  \delta : g,h \in G \mapsto \KantorovichDist{\Lip}{\alpha^g}{\alpha^h} \text{.}
\end{equation*}
In general, $\delta$ may only be a pseudo-metric, though it is induced by a pseudo-length function $\ell : g \in G\mapsto \KantorovichDist{\Lip}{\alpha^g}{\alpha^0}$ (where $\alpha^0$ is the identity automorphism of $\A$). We note that $\ell$ is a continuous function since $\KantorovichDist{\Lip}{\cdot}{\cdot}$ metrizes the topology of pointwise convergence in norm on the group of automorphisms of $\A$ by \cite{Latremoliere16b} and since $\alpha$ is strongly continuous.

Now, suppose $\ell(g) = 0$ for some $g\in G$, then for any $h\in G$ we note that:
\begin{equation*}
  \alpha^h \circ \alpha^g \circ \alpha^{h^{-1}} = \alpha^h\circ\alpha^0\circ\alpha^{h^{-1}} = \alpha^{0} 
\end{equation*}
so $K = \left\{ g \in G : \ell(g) = 0\right\}$ is closed by all inner automorphisms of $G$. It is then easy to check that $K$ is a subgroup, hence a normal subgroup, and it is closed by continuity of $\ell$. As a consequence, we can factor $\alpha$ through this subgroup to a free action of $H = \bigslant{G}{K}$ on $(\A,\Lip)$ by Lipschitz automorphisms. Note that $\ell$ induces a continuous length function on $H$ and $H$ is compact, so this length function induces the topology of $H$. We thus we obtain a class of systems whose right translations are Lipschitz, and the metric information is entirely from the dynamics and not the group, thus tying together the two metric structures which appear in our work --- the group metric and the metric from $\KantorovichDist{\cdot}{\cdot}{\cdot}$. This method does not directly apply to non-compact groups, since $\KantorovichDist{\cdot}{\cdot}{\cdot}$ is always bounded on the automorphism group \cite{Latremoliere16b} and a bounded proper metric space must be compact.


\bibliographystyle{amsplain}
\bibliography{../thesis}
\vfill

\end{document}